\documentclass[reqno]{amsart}	

\usepackage{enumerate}
\usepackage{graphicx}
\usepackage{nicefrac}
\usepackage{booktabs}
\usepackage{amsaddr, amsmath, amssymb}
\usepackage{stmaryrd}
\usepackage{caption}
\usepackage{subcaption}
\usepackage{tikz}
\usepackage{pgfplots}
\usepackage{geometry}
\usepackage{tabularx}
\usepackage{color}
\usepackage[onehalfspacing]{setspace}

\allowdisplaybreaks

\pgfplotsset{compat=1.15}
\definecolor{darkbrown}{rgb}{0.57, 0.40, 0.13}

\newcommand{\myspace}[1]{\mathbb{#1}}
\newcommand{\myvec}[1]{\mathfrak{#1}}          
\newcommand{\mymatrix}[1]{\boldsymbol{#1}}     

\newtheorem{theorem}{Theorem}

\newtheorem{corollary}[theorem]   {Corollary}

\newtheorem{lemma}[theorem]       {Lemma}

\theoremstyle{definition}

\newtheorem*{remark*}	{Remark}

\begin{document}

\title[Mixed Finite Elements of Higher-Order in Elastoplasticity]{Mixed Finite Elements of Higher-Order in Elastoplasticity}

\author[P.~Bammer, L.~Banz \and A.~Schr\"{o}der]{Patrick Bammer, Lothar Banz \and Andreas Schr\"{o}der}

\address{Fachbereich Mathematik, Paris Lodron Universit\"{a}t Salzburg,\\ Hellbrunner~Str.~14, 5020 Salzburg, Austria}

\begin{abstract}
In this paper a higher-order mixed finite element method for elastoplasticity with linear kinematic hardening is analyzed. Thereby, the non-differentiability of the involved plasticity functional is resolved by a Lagrange multiplier leading to a three field formulation. The finite element discretization is conforming in the displacement field and the plastic strain but potentially non-conforming in the Lagrange multiplier as its Frobenius norm is only constrained in a certain set of Gauss quadrature points. A discrete inf-sup condition with constant $1$ and the well posedness of the discrete mixed problem are shown. Moreover, convergence and guaranteed convergence rates are proved with respect to the mesh size and the polynomial degree, which are optimal for the lowest order case. Numerical experiments underline the theoretical results.
\end{abstract}


\thanks{%
A.~Schr\"{o}der acknowledges the support by the Bundesministerium f\"{u}r Bildung, Wissenschaft und Forschung (BMBWF) under the Sparkling Science project SPA 01-080 'MAJA -- Mathematische Algorithmen für Jedermann Analysiert'.
}


\maketitle


\section{Introduction}

Elastoplasticity with hardening appears in many problems of mechanical engineering, for instance in the modeling of the deformation of metal or concrete, see e.g.~\cite{chen1988plasticity}. The holonomic constitutive law represents a well established model for elastoplasticity with linear kinematic hardening and allows for the incremental computation of the deformation of an elastoplastic body, see e.g.~\cite{han1991finite, Han2013}. A (pseudo-)time step of that model can be formulated as a variational inequality of the second kind, in which a non-differentiable plasticity functional $\psi(\cdot)$ appears. The non-smoothness of $\psi(\cdot)$ causes many difficulties not only in the numerical analysis but also in the numerics. The latter include slow convergence of quadrature formulas for the evaluation of $\psi(\cdot)$ and also the need for non-standard and difficult to implement iterative solvers such as the Bundle-Newton method \cite{Luksan1998}. However, the use of such general iterative methods often causes an inefficient computation of a solution of the specific problem. One remedy is to regularize the plasticity functional as proposed in \cite{reddy1987variational} and to apply the standard Newton solver. Often, the required regularization parameters depend on the discretization. Choosing them too large may imply negative effects regarding the absolute value of the discretization error, whereas choosing them too small may significantly enlarge the condition number of the Newton iteration matrix (i.e.~the Jacobian matrix) as well as the number of iteration steps. 

To circumvent these difficulties resulting from the non-smoothness of $\psi(\cdot)$, one may reformulate the variational inequality as a mixed formulation in which the non-differentiability of the plasticity functional is resolved by a Lagrange multiplier, see e.g.~\cite{han1991finite, han1995finite,schroder2011error}, similar to frictional contact problems \cite{schroder2011mixed, zbMATH05872978}. The additional Lagrange multiplier can be discretized by a conforming approach with lowest order finite elements, see e.g.~\cite{schroder2011error}, or by a non-conforming higher-order method \cite{wiedemann2013Adaptive}. The algebraic systems resulting from the discretizations can be solved, for instance, by a simple to implement and super-linearly converging semi-smooth Newton method \cite{Bammer2022Icosahom}. We refer to \cite{brokate2005quasi, christensen2002nonsmooth, carstensen2006reliable, Han2013} for further details on solution schemes in elastoplasticity.

In this paper we provide an a priori error analysis (convergence and guaranteed convergence rates) for higher-order ($h$- and $p$-version) finite element discretizations of a mixed formulation for a model problem of elastoplasticity with linear kinematic hardening. The proposed discretization is defined on quasi-uniform meshes and is conforming with respect to the displacement field and the plastic strain. As the constraints for the Frobenius norm of the discrete Lagrange multiplier are only enforced in some Gauss quadrature points the discretization is non-conforming with respect to the Lagrange multiplier (except for the lowest order case). The non-conformity is necessary to receive an implementable discretization scheme but leads to a reduction of the guaranteed convergence rates by a factor of $2$ compared to the maximum rates enabled by the finite element spaces and compared to the experimental order of convergence. This reduction, however, is common for higher-order mixed methods for variational inequalities, see e.g.~\cite{banz2022priori,banz2019hybridization,ovcharova2017coupling}. 

We note that the solution of the model problem of elastoplacticity with hardening has typically a low regularity with a singularity at the free boundary which separates the regions of purely elastic deformations from the regions of plastic deformations. In order to achieve high convergence rates, a discretization with higher-order finite elements requires a combination of mesh refinements ($h$-refinements) in those regions, where the solution is singular, and an increase of the local polynomial degree ($p$-refinement), where the solution is locally smooth. Such $h$- and $hp$-refinements can be obtained, for instance, by using a posteriori error control and adaptivity. We refer to~\cite{Banz2021abstract, burg2015posteriori} for a posteriori error control and adaptivity for variational inequalities related to contact problems and to \cite{zbMATH06537874} in the context of elastoplasticity. \\
 
The paper is structured as follows: In Section~\ref{BBS_sec:notation} the strong formulation of the model problem of elastoplasticity with linear kinematic hardening as well as its weak formulation as a variational inequality of the second kind are presented. The equivalent mixed formulation and some of its main properties can be found in Section~\ref{BBS_subsec:MVF}. Section~\ref{BBS_sec:discretization} is devoted to the higher-order finite element discretization of the mixed formulation, which is easy to realize as the same mesh for all three variables and even the same basis functions for the discrete plastic strain and the discrete Lagrange multiplier are used. Moreover, a discrete inf-sup condition with constant $1$ and the well posedness of the discrete mixed problem are shown. The main results, namely the convergence of the method (Theorem~\ref{BBS_thm:convergence_mixed_Lambda_hp}) and the guaranteed convergence rates in $h$ and $p$ (Theorem~\ref{BBS_thm:convergence_rates_mixedForm}) can be found in Section~\ref{BBS_sec:a_priori_analysis}. Furthermore, for the lowest order case the optimal convergence rates  are proved (Theorem~\ref{thm:conver_p1}). Finally, the numerical results of Section~\ref{BBS_sec:numerical_results} underline the theoretical findings and, additionally, the potential of $h$- and $hp$-adaptive refinements.

\medskip


\section{The Model Problem and Notation}\label{BBS_sec:notation}

Let $\Omega \subset \myspace{R}^d$ with $d \in \lbrace 2,3 \rbrace$ be a bounded, polygonal domain with Lipschitz-boundary $\Gamma := \partial \Omega$ and outer unit normal $\myvec{n}$ representing the reference configuration of an elastoplastic body. The body is clamped at a Dirchlet boundary part $\Gamma_D \subseteq \Gamma$ of positive surface measure and is subjected to some volume force $\myvec{f}$ and some surface force $\myvec{g}$ on the body and the Neumann boundary part $\Gamma_N := \Gamma \setminus \overline{\Gamma_D}$, respectively. The model problem of elastoplasticity with linear kinematic hardening, see e.g.~\cite{Han2013}, is to find a displacement field $\myvec{u} \in H^1(\Omega,\myspace{R}^d)$ and a plastic strain $\mymatrix{p} \in L^2(\Omega, \myspace{S}_{d,0})$, where
\begin{align*}
 \myspace{S}_{d,0}  := \bigg\lbrace \mymatrix{\tau} \in \myspace{R}^{d \times d}  \; ; \; \mymatrix{\tau} = \mymatrix{\tau}^{\top}, \;\operatorname{tr}(\mymatrix{\tau}) := \sum_{i=1}^d \tau_{ii} = 0 \bigg\rbrace,
\end{align*}
such that there holds
\pagebreak
\begin{subequations}\label{BBS_eq:model_problem}
\begin{alignat}{2}
 -\operatorname{div} \mymatrix{\sigma}(\myvec{u},\mymatrix{p}) &= \myvec{f}   & \quad & \text{in } \Omega, \label{BBS_eq:model_PDE} \\
 \myvec{u} &= \myvec{o}                                            & \quad & \text{on } \Gamma_D, \\
 \mymatrix{\sigma}(\myvec{u},\mymatrix{p}) \, \myvec{n} &= \myvec{g}  & \quad & \text{on } \Gamma_N, \label{BBS_eq:model_boundaryCond} \\
 \mymatrix{\sigma}(\myvec{u},\mymatrix{p})-\myspace{H} \mymatrix{p} &\in \partial j(\mymatrix{p}) & \quad &  \text{in } \Omega. \label{BBS_eq:differntial_inclusion}
\end{alignat}
\end{subequations}
Here, $\mymatrix{\sigma}(\myvec{u},\mymatrix{p}) := \myspace{C}(\mymatrix{\varepsilon}(\myvec{u})-\mymatrix{p})$ is the stress tensor, $\mymatrix{\varepsilon}(\myvec{u}) := \frac{1}{2} \, \big(\nabla\myvec{u} + (\nabla\myvec{u})^{\top} \big)$ is the linearized strain tensor, and $\myspace{C}$ and $\myspace{H}$ denote the fourth-order elasticity and hardening tensor, respectively. We assume that the entries of $\myspace{C}$ and $\myspace{H}$ are in $L^{\infty}(\Omega)$ and that $\myspace{C}$ and $\myspace{H}$ are symmetric and uniformly elliptic, i.e.~$\myspace{C}_{ijkl}=\myspace{C}_{jilk} = \myspace{C}_{kli j}$, $\myspace{H}_{ijkl}=\myspace{H}_{jilk} = \myspace{H}_{klij}$ for all $1\leq i,j,k,l\leq d$ and there exist positive constants $c_e,c_h>0$ such that
\begin{align*}
 (\myspace{C}\mymatrix{\tau}):\mymatrix{\tau} \geq c_e \, \vert\mymatrix{\tau}\vert_F^2 ,
   \quad
 (\myspace{H}\mymatrix{\tau}):\mymatrix{\tau} \geq c_h \, \vert\mymatrix{\tau}\vert_F^2 \qquad 
 \forall \,\mymatrix{\tau}\in L^2(\Omega,\myspace{S}_d).
\end{align*}
As usual, $\vert\mymatrix{\tau}\vert_F^2 = \mymatrix{\tau} : \mymatrix{\tau}$ stands for the Frobenius norm squared and $:$ for the Frobenius inner product. Furthermore, $\partial j(\cdot)$ represents the subdifferential of the non-differentiable part of the plastic dissipation functional $j(\cdot)$, which is given by $j(\mymatrix{q}) := \sigma_y \, \vert\mymatrix{q}\vert_F$ with a constant yield stress in uniaxial tension $\sigma_y>0$. 

\begin{remark*}
Replacing the plastic strain $\mymatrix{p}$ in $\partial j(\mymatrix{p})$ in the line \eqref{BBS_eq:differntial_inclusion} by its time derivative $\dot{\mymatrix{p}} := \frac{\partial \mymatrix{p}}{\partial t}$ the formulation \eqref{BBS_eq:model_problem} describes one time step for a quasi-static elastoplasticity problem with hardening and initial conditions $\mymatrix{p}_0 := \mymatrix{0}$, see \cite{carstensen2006reliable, han1991finite, Han2013}. In fact, by updating the right hand side data, the formulation of any time step takes the form \eqref{BBS_eq:model_problem}.
\end{remark*}

Here and in the following we use standard letters such as $v$ for scalar functions, fraktur font $\myvec{v}$ for vector valued functions and bold font $\mymatrix{v}$ for matrix valued functions. Moreover, capital letters indicate vector spaces and subsets of them. We define
\begin{align*}
 V := \Big\lbrace \myvec{v} \in H^1(\Omega, \myspace{R}^d) \; ; \; \myvec{v}_{\, | \, \Gamma_D} = \myvec{o} \Big\rbrace,
 \qquad \qquad
 Q := L^2(\Omega,\myspace{S}_{d,0})
\end{align*}
and note that Korn's inequality allows equipping $V$ with the norm $\Vert \myvec{v}\Vert_{1,\Omega} := \big(\Vert \myvec{v} \Vert_{0,\Omega}^2 + |\myvec{v}|_{1,\Omega}^2\big)^{1/2}$ where $|\myvec{v}|_{1,\Omega}^2 := \big(\mymatrix{\varepsilon}(\myvec{v}),\mymatrix{\varepsilon}(\myvec{v})\big)_{0,\Omega}$. Thereby, we denote the $L^2(\Omega)$-inner product by $(\cdot,\cdot)_{0,\Omega}$ and its induced $L^2(\Omega)$-norm by $\Vert \cdot \Vert_{0,\Omega}$ (for scalar, vector as well as matrix valued functions). Likewise, we denote the duality pairing between the involved Hilbert spaces by $\langle\cdot,\cdot\rangle$, where an addition index can be used to emphasize the associated domain. Moreover, $\widetilde{H}^{1/2}(\Gamma_N,\myspace{R}^d)$ is the trace space of $V$ restricted to the boundary part $\Gamma_N$ equipped with the standard trace norm $\Vert\cdot\Vert_{\frac{1}{2},\Gamma_N}$, and $H^{-1/2}(\Gamma_N,\myspace{R}^d)$ is its dual space with the dual norm $\Vert\cdot\Vert_{-\frac{1}{2},\Gamma_N}$. Recall that by the trace theorem there holds
\begin{align*}
 \Vert \myvec{v}\Vert_{\frac{1}{2}, \Gamma_N}
 \leq c_{tr} \, \Vert \myvec{v}\Vert_{1,\Omega}
 \qquad \forall \, \myvec{v}\in V
\end{align*}
for some constant $c_{tr}>0$.
The deviatoric part of $\mymatrix{\tau}\in\myspace{R}^{d\times d}$ is given by $\operatorname{dev}(\mymatrix{\tau}) := \mymatrix{\tau} - \frac{1}{d}\operatorname{tr}(\mymatrix{\tau}) \, \mymatrix{I}$ where $\mymatrix{I} \in \myspace{R}^{d\times d}$ is the identity matrix.
From $\operatorname{tr}(\mymatrix{q}) = 0$ for any $\mymatrix{q}\in Q$ we obtain
\begin{align}\label{BBS_eq:identity_deviatoricPart}
 \big( \operatorname{dev}(\mymatrix{\mu}), \mymatrix{q} \big)_{0,\Omega}
 = ( \mymatrix{\mu}, \mymatrix{q})_{0,\Omega} - \frac{1}{d} \, \big( \operatorname{tr}(\mymatrix{\mu}) \, \mymatrix{I}, \mymatrix{q} \big)_{0,\Omega}
 = ( \mymatrix{\mu}, \mymatrix{q})_{0,\Omega} - \frac{1}{d} \, \big( \operatorname{tr}(\mymatrix{\mu}), \operatorname{tr}(\mymatrix{q}) \big)_{0,\Omega}
 = ( \mymatrix{\mu}, \mymatrix{q})_{0,\Omega}
\end{align}
for all $\mymatrix{q}\in Q$ and $\mymatrix{\mu}\in L^2(\Omega,\myspace{R}^{d\times d})$.

It is well known, see e.g.~\cite{carstensen1999numerical, han1991finite, Han2013}, that for $\myvec{f}\in V^*$ and $\myvec{g}\in H^{-1/2}(\Gamma_N,\myspace{R}^d)$ the problem \eqref{BBS_eq:model_problem} has a weak formulation in the form of a variational inequality of the second kind. 
For this purpose, let the bilinear form $a(\cdot,\cdot)$, the plasticity functional $\psi(\cdot)$ and the linear form $\ell(\cdot)$ be given by
\begin{align*}
 a\big( (\myvec{u},\mymatrix{p}),(\myvec{v},\mymatrix{q}) \big) 
 &:= \big( \mymatrix{\sigma}(\myvec{u},\mymatrix{p}), \mymatrix{\varepsilon}(\myvec{v})-\mymatrix{q} \big)_{0,\Omega} + (\myspace{H} \mymatrix{p}, \mymatrix{q})_{0,\Omega},\\
 \psi(\mymatrix{q}) 
 &:=(\sigma_y,\vert \mymatrix{q}\vert_F)_{0,\Omega},\\
 \ell(\myvec{v}) 
 &:= \langle \myvec{f},\myvec{v} \rangle_{\Omega} + \langle \myvec{g},\myvec{v} \rangle_{\Gamma_N},
\end{align*}
respectively. Then, the variational inequality problem is to find a pair $(\myvec{u},\mymatrix{p})\in V\times Q$ such that
\begin{align}\label{BBS_eq:variq_second_kind}
 a\big( (\myvec{u},\mymatrix{p}), (\myvec{v}-\myvec{u},\mymatrix{q}-\mymatrix{p}) \big) + \psi(\mymatrix{q}) - \psi(\mymatrix{p}) 
 \geq \ell(\myvec{v}-\myvec{u}) \qquad 
 \forall\, (\myvec{v},\mymatrix{q})\in V\times Q.
\end{align}
Equipping the Hilbert space $V\times Q$ with the norm $\Vert (\myvec{v},\mymatrix{q}) \Vert^2 := \Vert \myvec{v}\Vert_{1,\Omega}^2 + \Vert \mymatrix{q}\Vert_{0,\Omega}^2$ for some $(\myvec{v},\mymatrix{q})\in V\times Q$ it is easy to see, cf.~\cite{Han2013}, that the bilinear form $a(\cdot,\cdot)$ is symmetric, continuous and $(V\times Q)$-elliptic, i.e.~there exist positive constants $c_a$, $\alpha>0$ such that
\begin{align*}
 a\big( (\myvec{u},\mymatrix{p}),(\myvec{v},\mymatrix{q}) \big) 
 \leq c_a \, \Vert (\myvec{u},\mymatrix{p})\Vert \, \Vert (\myvec{v},\mymatrix{q}) \Vert,
  \quad
 \alpha \, \Vert (\myvec{v},\mymatrix{q})\Vert^2 
 \leq a\big( (\myvec{v},\mymatrix{q}),(\myvec{v},\mymatrix{q}) \big) \qquad 
 \forall \, (\myvec{u},\mymatrix{p}), (\myvec{v},\mymatrix{q})\in V\times Q.
\end{align*}
Obviously, the plasticity functional $\psi(\cdot)$ is convex and Lipschitz continuous as
\begin{align*}
 | \psi(\mymatrix{p}) - \psi(\mymatrix{q}) | 
 \leq \Vert\sigma_y\Vert_{0,\Omega} \, \big\Vert \, |\mymatrix{p}|_F- |\mymatrix{q}|_F \big\Vert_{0,\Omega}
 \leq \Vert\sigma_y \Vert_{0,\Omega} \, \big\Vert \, |\mymatrix{p}- \mymatrix{q}|_F \big\Vert_{0,\Omega} 
 = \Vert\sigma_y \Vert_{0,\Omega} \, \Vert \mymatrix{p}-\mymatrix{q}\Vert_{0,\Omega} \qquad 
 \forall\, \mymatrix{p},\mymatrix{q}\in Q.
\end{align*}
By using the Cauchy-Schwarz inequality it is easy to verify that
\begin{align*}
 \psi(\mymatrix{q}) \geq \psi(\mymatrix{p}) + ( \widetilde{\mymatrix{q}}, \mymatrix{q} - \mymatrix{p} )_{0,\Omega} \qquad 
 \forall \, \mymatrix{p},\mymatrix{q}\in Q
\end{align*}
with the function $\widetilde{\mymatrix{q}} \in Q$ defined as
\begin{align*}
 \widetilde{\mymatrix{q}} := \begin{cases} 
 \sigma_y \, \vert\mymatrix{p}\vert_F^{-1} \, \mymatrix{p}, & \text{wherever } \mymatrix{p} \neq \mymatrix{0}, \\ 
 \mymatrix{0}, & \text{elsewhere} . \end{cases}
\end{align*}
Hence, $\psi(\cdot)$ is subdifferentiable. It is well known, see e.g.~\cite{Han2013}, that the variational inequality \eqref{BBS_eq:variq_second_kind} is equivalent to the strictly convex but non-smooth minimization problem: Find a pair $(\myvec{u},\mymatrix{p})\in V\times Q$ such that
\begin{align}\label{BBS_eq:minimization_problem}
 \mathcal{E}(\myvec{u},\mymatrix{p}) 
 \leq \mathcal{E}(\myvec{v},\mymatrix{q}) \qquad 
 \forall \, (\myvec{v},\mymatrix{q})\in V\times Q
\end{align}
with the energy functional $\mathcal{E} : V\times Q \rightarrow \myspace{R}$ given by
\begin{align*}
 \mathcal{E}(\myvec{v},\mymatrix{q}) 
 := \frac{1}{2} \, a\big( (\myvec{v},\mymatrix{q}),(\myvec{v},\mymatrix{q}) \big) + \psi(\mymatrix{q}) - \ell(\myvec{v}).
\end{align*}

\medskip


\section{A Mixed Variational Formulation}\label{BBS_subsec:MVF}

As $\mathcal{E}(\cdot)$ is coercive, convex and weakly lower semicontinuous the minimization problem \eqref{BBS_eq:minimization_problem} and, thus, the variational inequality \eqref{BBS_eq:variq_second_kind} have a unique solution, see~\cite{Han2013}. In order to deal with the non-differentiability of $\psi(\cdot)$ we introduce a Lagrange multiplier resolving the Frobenius norm in the definition of the plasticity functional $\psi(\cdot)$. To this end, we consider the following non-empty, closed and convex set
\begin{align*}
 \Lambda 
 := \big\lbrace \mymatrix{\mu}\in Q \; ; \; \vert \mymatrix{\mu}\vert_F \leq \sigma_y \text{ a.e.~in } \Omega \big\rbrace.
\end{align*}
For any $\mymatrix{\mu}\in \Lambda$ the Cauchy-Schwarz inequality yields
\begin{align}\label{BBS_eq:mu_in_LambdaW}
 (\mymatrix{\mu},\mymatrix{q})_{0,\Omega} 
 \leq \int_{\Omega} \vert\mymatrix{\mu} \vert_F \, \vert \mymatrix{q} \vert_F \mathrm{\, d} \myvec{x}  
 \leq \int_{\Omega} \sigma_y \, \vert \mymatrix{q}\vert_F \mathrm{\, d} \myvec{x}
 = (\sigma_y, \vert\mymatrix{q}\vert_F)_{0,\Omega} \qquad 
 \forall \, \mymatrix{q}\in Q.
\end{align}
Conversely, for any $\mymatrix{\mu}\in Q$ satisfying \eqref{BBS_eq:mu_in_LambdaW} let $\overline{\mymatrix{q}} :=  \vert\mymatrix{\mu}\vert_F^{-1} \mymatrix{\mu} $ on $N_{\mymatrix{\mu}} := \big\lbrace \myvec{x}\in\Omega \; ; \; \vert \mymatrix{\mu}(\myvec{x})\vert_F > \sigma_y \big\rbrace$ and $\overline{\mymatrix{q}} := \mymatrix{0}$ elsewhere. Then, $\overline{\mymatrix{q}}\in Q$ and the inequality \eqref{BBS_eq:mu_in_LambdaW} immediately leads to
\begin{align*}
 0 
 \leq \int_{\Omega} \sigma_y \, \vert \overline{\mymatrix{q}}\vert_F - \mymatrix{\mu} : \overline{\mymatrix{q}} \mathrm{\, d} \myvec{x}
 = \int_{N_{\mymatrix{\mu}}} \underbrace{\sigma_y - \vert \mymatrix{\mu}\vert_F}_{<0} \mathrm{\, d} \myvec{x},
\end{align*}
which implies that $N_{\mymatrix{\mu}}$ has to be a set of measure zero. Thus, $\Lambda$ can alternatively be written as
\begin{align}\label{BBS_eq:characterization_Lambda}
 \Lambda 
 = \big\lbrace \mymatrix{\mu}\in Q \; ; \; (\mymatrix{\mu},\mymatrix{q})_{0,\Omega} \leq \psi(\mymatrix{q}) \text{ for all } \mymatrix{q}\in Q \big\rbrace
\end{align}
and we easily find the relation
\begin{align} \label{eq:sup_def_psi}
 \psi(\mymatrix{q}) 
 = \sup_{\mymatrix{\mu}\in\Lambda} (\mymatrix{\mu},\mymatrix{q})_{0,\Omega}.
\end{align}
Thereby, a mixed variational formulation to \eqref{BBS_eq:model_problem} is to find a triple $(\myvec{u},\mymatrix{p},\mymatrix{\lambda})\in V\times Q\times \Lambda$ such that
\begin{subequations}\label{BBS_eq:mixed_variationalF}
\begin{alignat}{2}
 a\big( (\myvec{u},\mymatrix{p}),(\myvec{v},\mymatrix{q}) \big) + (\mymatrix{\lambda}, \mymatrix{q})_{0,\Omega} &= \ell(\myvec{v}) & \qquad & \forall \, (\myvec{v},\mymatrix{q}) \in V \times Q, \label{BBS_eq:mixed_variationalF_01}\\
 (\mymatrix{\mu}-\mymatrix{\lambda}, \mymatrix{p})_{0,\Omega} &\leq 0 & \quad & \forall \, \mymatrix{\mu} \in \Lambda . \label{BBS_eq:mixed_variationalF_02}
\end{alignat}
\end{subequations}
The existence of a unique solution to \eqref{BBS_eq:mixed_variationalF} is guaranteed by the next result, which shows the equivalence of the mixed variational formulation \eqref{BBS_eq:mixed_variationalF} to the variational inequality \eqref{BBS_eq:variq_second_kind}.

\begin{theorem}\label{BBS_thm:equivalence_continuous_level}
If $(\myvec{u},\mymatrix{p})\in V\times Q$ solves \eqref{BBS_eq:variq_second_kind}, then $(\myvec{u},\mymatrix{p},\mymatrix{\lambda})$ with $\mymatrix{\lambda} := \operatorname{dev}(\mymatrix{\sigma}(\myvec{u},\mymatrix{p}) - \myspace{H}\mymatrix{p})$ is a solution of \eqref{BBS_eq:mixed_variationalF}. Conversely, if $(\myvec{u},\mymatrix{p},\mymatrix{\lambda})\in V\times Q\times \Lambda$ solves \eqref{BBS_eq:mixed_variationalF}, then $(\myvec{u},\mymatrix{p})$ is a solution of \eqref{BBS_eq:variq_second_kind} and there holds $\mymatrix{\lambda} = \operatorname{dev}(\mymatrix{\sigma}(\myvec{u},\mymatrix{p}) - \myspace{H}\mymatrix{p})$ a.e.~in $\Omega$.
Moreover, in both cases the following complementarity condition is satisfied
\begin{align} \label{eq:complementarityCondition}
 (\mymatrix{\lambda},\mymatrix{p})_{0,\Omega}
 = \psi(\mymatrix{p}).
\end{align}
\end{theorem}

\begin{proof}
Let $(\myvec{u},\mymatrix{p})\in V\times Q$ be the solution of \eqref{BBS_eq:variq_second_kind} and $\mymatrix{\lambda} := \operatorname{dev}(\mymatrix{\sigma}(\myvec{u},\mymatrix{p}) - \myspace{H}\mymatrix{p})$. Obviously, there holds $\mymatrix{\lambda}\in Q$ by the symmetry assumptions on $\myspace{C}$ and $\myspace{H}$, and the definition of the deviatoric part of a matrix. Choosing $\mymatrix{p} \in Q$ and $\myvec{u}\pm\myvec{v}$ for some $\myvec{v} \in V$ as test functions in \eqref{BBS_eq:variq_second_kind} yields 
\begin{align}\label{BBS_eq:lame_equation}
 \big( \mymatrix{\sigma}(\myvec{u},\mymatrix{p}), \mymatrix{\varepsilon}(\myvec{v}) \big)_{0,\Omega}
 = \ell(\myvec{v}) \qquad 
 \forall\, \myvec{v}\in V.
\end{align}
Thus, the definition of the bilinear form $a(\cdot,\cdot)$ and \eqref{BBS_eq:identity_deviatoricPart} show that
\begin{align}\label{BBS_eq:proof_thm1_a}
 \ell(\myvec{v}) 
 &= a\big( (\myvec{u},\mymatrix{p}), (\myvec{v},\mymatrix{q}) \big) + \big( \mymatrix{\sigma}(\myvec{u},\mymatrix{p}) - \myspace{H}\mymatrix{p}, \mymatrix{q} \big)_{0,\Omega} 
 = a\big( (\myvec{u},\mymatrix{p}), (\myvec{v},\mymatrix{q}) \big) + (\mymatrix{\lambda},\mymatrix{q})_{0,\Omega} \qquad 
 \forall \, (\myvec{v},\mymatrix{q})\in V\times Q,
\end{align}
which is \eqref{BBS_eq:mixed_variationalF_01}. In particular, choosing $(\myvec{0},\mymatrix{q}-\mymatrix{p})\in V\times Q$ as test function in \eqref{BBS_eq:proof_thm1_a} gives
\begin{align} \label{BBS_eq:proof_thm_1}
 (\mymatrix{\lambda},\mymatrix{q}-\mymatrix{p})_{0,\Omega}
 = - a\big( (\myvec{u},\mymatrix{p}), (\myvec{0},\mymatrix{q}-\mymatrix{p}) \big) \leq \psi(\mymatrix{q}) - \psi(\mymatrix{p}) \qquad 
 \forall \, \mymatrix{q}\in Q,
\end{align}
where the inequality follows from \eqref{BBS_eq:variq_second_kind}. Now, choosing $\mymatrix{q}=\mymatrix{0}$ and $\mymatrix{q}=2\mymatrix{p}$ in \eqref{BBS_eq:proof_thm_1} proves the complementarity condition \eqref{eq:complementarityCondition}. In particular, \eqref{eq:complementarityCondition} and \eqref{BBS_eq:proof_thm_1} imply $(\mymatrix{\lambda},\mymatrix{q})_{0,\Omega} \leq \psi(\mymatrix{q})$ for any $\mymatrix{q}\in Q$ and, thus, $\mymatrix{\lambda}\in\Lambda$. The equivalent characterization \eqref{BBS_eq:characterization_Lambda} of $\Lambda$, $\mymatrix{\lambda}\in\Lambda$ and \eqref{eq:complementarityCondition} imply
\begin{align*}
 (\mymatrix{\mu}-\mymatrix{\lambda},\mymatrix{p})_{0,\Omega}
 = (\mymatrix{\mu},\mymatrix{p})_{0,\Omega} - \psi(\mymatrix{p}) 
 \leq 0 \qquad 
 \forall \, \mymatrix{\mu}\in\Lambda,
\end{align*}
which is \eqref{BBS_eq:mixed_variationalF_02}. Hence, the triple $(\myvec{u},\mymatrix{p},\mymatrix{\lambda})$ solves the mixed variational problem \eqref{BBS_eq:mixed_variationalF}.\\

Conversely, let $(\myvec{u},\mymatrix{p},\mymatrix{\lambda})\in V\times Q \times \Lambda$ solve \eqref{BBS_eq:mixed_variationalF}. Choosing $\mymatrix{q}=\mymatrix{0}$ in \eqref{BBS_eq:mixed_variationalF_01} yields \eqref{BBS_eq:lame_equation}. Therefore, from \eqref{BBS_eq:mixed_variationalF_01} and by using \eqref{BBS_eq:identity_deviatoricPart} we obtain
\begin{align*}
 0 
 &= \big( \mymatrix{\sigma}(\myvec{u},\mymatrix{p}), \mymatrix{\varepsilon}(\myvec{v}) \big)_{0,\Omega} - a\big( (\myvec{u},\mymatrix{p}),(\myvec{v},\mymatrix{q}) \big) - (\mymatrix{\lambda}, \mymatrix{q})_{0,\Omega} \\
 &= \big( \mymatrix{\sigma}(\myvec{u},\mymatrix{p}) - \myspace{H}\mymatrix{p}, \mymatrix{q} \big)_{0,\Omega} - (\mymatrix{\lambda},\mymatrix{q})_{0,\Omega} \\
 &= \big( \operatorname{dev}(\mymatrix{\sigma}(\myvec{u},\mymatrix{p}) - \myspace{H}\mymatrix{p}) - \mymatrix{\lambda}, \mymatrix{q} \big)_{0,\Omega}
\end{align*}
for any $\mymatrix{q}\in Q$ from which we deduce $\mymatrix{\lambda} = \operatorname{dev}(\mymatrix{\sigma}(\myvec{u},\mymatrix{p}) - \myspace{H}\mymatrix{p})\in Q$. From \eqref{BBS_eq:mixed_variationalF_02} and \eqref{BBS_eq:characterization_Lambda} we obtain
\begin{align*}
 (\mymatrix{\mu},\mymatrix{p})_{0,\Omega}
 \leq (\mymatrix{\lambda},\mymatrix{p})_{0,\Omega}
 \leq \psi(\mymatrix{p}) \qquad 
 \forall \, \mymatrix{\mu}\in\Lambda,
\end{align*}
which together with \eqref{eq:sup_def_psi} implies that
\begin{align*}
 \psi(\mymatrix{p}) 
 = \sup_{\mymatrix{\mu}\in\Lambda} (\mymatrix{\mu},\mymatrix{p})_{0,\Omega} 
 = (\mymatrix{\lambda},\mymatrix{p})_{0,\Omega},
\end{align*}
i.e.~there holds \eqref{eq:complementarityCondition}.
Using $\myvec{v}-\myvec{u}$ and $\mymatrix{q}-\mymatrix{p}$ with arbitrary $(\myvec{v},\mymatrix{q})\in V\times Q$ as test functions in \eqref{BBS_eq:mixed_variationalF_01} we find that
\begin{align*}
 \ell(\myvec{v}-\myvec{u})
 &= a\big( (\myvec{u},\mymatrix{p}),(\myvec{v}-\myvec{u},\mymatrix{q}-\mymatrix{p}) \big) + (\mymatrix{\lambda}, \mymatrix{q}-\mymatrix{p})_{0,\Omega} \\
 &= a\big( (\myvec{u},\mymatrix{p}),(\myvec{v}-\myvec{u},\mymatrix{q}-\mymatrix{p}) \big) + (\mymatrix{\lambda},\mymatrix{q})_{0,\Omega} - \psi(\mymatrix{p}) \\
 &\leq a\big( (\myvec{u},\mymatrix{p}),(\myvec{v}-\myvec{u},\mymatrix{q}-\mymatrix{p}) \big) + \psi(\mymatrix{q}) - \psi(\mymatrix{p})
\end{align*}
as $\mymatrix{\lambda}\in\Lambda$. Hence, $(\myvec{u},\mymatrix{p})$ solves the variational inequality \eqref{BBS_eq:variq_second_kind}.
\end{proof}

\begin{corollary}\label{BBS_lem:properties_lagrangeMult}
It holds $\mymatrix{p} : \mymatrix{\lambda} = \sigma_y \, \vert \mymatrix{p}\vert_F$ a.e.~in $\Omega$.
\end{corollary}

\begin{proof}
The Cauchy-Schwarz inequality together with $\mymatrix{\lambda}\in\Lambda$ immediately gives
\begin{align*}
 0 
 = \sigma_y \, \vert\mymatrix{p}\vert_F - \sigma_y \, \vert\mymatrix{p}\vert_F 
 \leq \sigma_y \, \vert\mymatrix{p}\vert_F - \vert\mymatrix{\lambda}\vert_F \, \vert\mymatrix{p}\vert_F 
 \leq  \sigma_y \, \vert\mymatrix{p}\vert_F - \mymatrix{\lambda} : \mymatrix{p}
\end{align*}
a.e.~in $\Omega$. Now, the assertion follows with the complementarity condition \eqref{eq:complementarityCondition}.
\end{proof}

\begin{lemma}
The solution of \eqref{BBS_eq:mixed_variationalF} depends Lipschitz-continuously on the data $\myvec{f}$, $\myvec{g}$ and $\sigma_y$. More precisely,
\begin{align*}
 \Vert (\myvec{u}_2-\myvec{u}_1, \mymatrix{p}_2-\mymatrix{p}_1) \Vert + \Vert \mymatrix{\lambda}_2-\mymatrix{\lambda}_1 \Vert_{0,\Omega} 
 \leq \frac{1+c_a}{\alpha} \, \Big( \Vert \sigma_{y,2} -\sigma_{y,1} \Vert_{0,\Omega} + \Vert \myvec{f}_2-\myvec{f}_1\Vert_{V^*} + c_{tr} \, \Vert \myvec{g}_2-\myvec{g}_1\Vert_{-\frac{1}{2},\Gamma_N} \Big),
\end{align*}
where $(\myvec{u}_i,\mymatrix{p}_i,\mymatrix{\lambda}_i)$ is the solution to the data $(\myvec{f}_i,\myvec{g}_i,\sigma_{y,i})$ for $i=1,2$.
\end{lemma}

\begin{proof}
Let $\ell_i(\cdot) := \langle \myvec{f}_i,\cdot \rangle_{\Omega} + \langle \myvec{g}_i,\cdot \rangle_{\Gamma_N}$ and $\psi_i(\cdot) := (\sigma_{y,i}, \vert\cdot\vert_F)_{0,\Omega}$. Furthermore, define
\begin{align*}
 \Lambda_i 
 := \big\lbrace \mymatrix{\mu}\in Q \; ; \; (\mymatrix{\mu},\mymatrix{q})_{0,\Omega} \leq \psi_i(\mymatrix{q}) \text{ for all } \mymatrix{q}\in Q \big\rbrace
\end{align*}
and let $(\myvec{u}_i,\mymatrix{p}_i,\mymatrix{\lambda}_i)\in V\times Q\times \Lambda_i$ be the unique solution to
\begin{subequations}
\begin{alignat}{2}
 a\big( (\myvec{u}_i,\mymatrix{p}_i),(\myvec{v},\mymatrix{q}) \big) + (\mymatrix{\lambda}_i, \mymatrix{q})_{0,\Omega} &= \ell_i(\myvec{v}) & \qquad & \forall \, (\myvec{v},\mymatrix{q}) \in V \times Q, \label{BBS_eq:proofLEM4_vareq_01} \\
 (\mymatrix{\mu}-\mymatrix{\lambda}_i,\mymatrix{p}_i)_{0,\Omega} &\leq 0 & \quad & \forall \, \mymatrix{\mu} \in \Lambda_i.
\end{alignat}
\end{subequations}
Subtracting the equations \eqref{BBS_eq:proofLEM4_vareq_01} for $i=1,2$ from each other yields
\begin{align}\label{BBS_eq:proofLEM4_vareq_02}
 a\big( (\myvec{u}_2-\myvec{u}_1,\mymatrix{p}_2-\mymatrix{p}_1),(\myvec{v},\mymatrix{q}) \big) + (\mymatrix{\lambda}_2 - \mymatrix{\lambda}_1, \mymatrix{q}) _{0,\Omega}
 &= \ell_2(\myvec{v}) - \ell_1(\myvec{v}) \qquad  
 \forall \, (\myvec{v},\mymatrix{q}) \in V \times Q
\end{align}
and, in particular,
\begin{align}\label{BBS_eq:proofLEM4_vareq_03}
 a\big( (\myvec{u}_2-\myvec{u}_1,\mymatrix{p}_2-\mymatrix{p}_1), (\myvec{u}_2-\myvec{u}_1,\mymatrix{p}_2-\mymatrix{p}_1) \big)
 = \ell_2(\myvec{u}_2-\myvec{u}_1) - \ell_1(\myvec{u}_2-\myvec{u}_1) - (\mymatrix{\lambda}_2 - \mymatrix{\lambda}_1, \mymatrix{p}_2-\mymatrix{p}_1)_{0,\Omega}.
\end{align}
Exploiting the complementarity condition \eqref{eq:complementarityCondition} and using that $\mymatrix{\lambda}_i\in\Lambda_i$ we find that
\begin{align*}
 - (\mymatrix{\lambda}_2 - \mymatrix{\lambda}_1, \mymatrix{p}_2-\mymatrix{p}_1)_{0,\Omega}
  &= (\mymatrix{\lambda}_2, \mymatrix{p}_1)_{0,\Omega} - (\mymatrix{\lambda}_2, \mymatrix{p}_2)_{0,\Omega} + (\mymatrix{\lambda}_1, \mymatrix{p}_2)_{0,\Omega} - (\mymatrix{\lambda}_1, \mymatrix{p}_1)_{0,\Omega} \\
  &\leq \psi_2(\mymatrix{p}_1) - \psi_2(\mymatrix{p}_2) + \psi_1(\mymatrix{p}_2) - \psi_1(\mymatrix{p}_1) \\
  &= \big( \sigma_{y,2}-\sigma_{y,1}, \vert \mymatrix{p}_1\vert_F - \vert \mymatrix{p}_2\vert_F \big)_{0,\Omega} \\
  &\leq \Vert \sigma_{y,2}-\sigma_{y,1}\Vert_{0,\Omega} \, \Vert \mymatrix{p}_1 - \mymatrix{p}_2\Vert_{0,\Omega},
\end{align*}
where the last inequality follows from the Cauchy-Schwarz inequality and the reverse triangle inequality, which gives
\begin{align*} 
 \big\Vert \vert \mymatrix{p}_1\vert_F - \vert\mymatrix{p}_2\vert_F \big\Vert_{0,\Omega} 
 = \left( \int_{\Omega} \big( \vert \mymatrix{p}_1\vert_F - \vert\mymatrix{p}_2\vert_F \big)^2 \mathrm{\, d}\myvec{x} \right)^{1/2}
 \leq \left( \int_{\Omega} \vert \mymatrix{p}_1 - \mymatrix{p}_2\vert_F^2 \mathrm{\, d}\myvec{x} \right)^{1/2}
 = \Vert \mymatrix{p}_1 - \mymatrix{p}_2\Vert_{0,\Omega}.
\end{align*}
Furthermore, by using the trace theorem we obtain
\begin{align*}
 \ell_2(\myvec{u}_2-\myvec{u}_1) - \ell_1(\myvec{u}_2-\myvec{u}_1)
 &= \langle \myvec{f}_2-\myvec{f}_1, \myvec{u}_2-\myvec{u}_1 \rangle_{\Omega} + \langle \myvec{g}_2-\myvec{g}_1, \myvec{u}_2-\myvec{u}_1 \rangle_{\Gamma_N} \\
 &\leq \Vert \myvec{f}_2-\myvec{f}_1 \Vert_{V^*} \, \Vert \myvec{u}_2-\myvec{u}_1 \Vert_{1,\Omega} + \Vert \myvec{g}_2-\myvec{g}_1 \Vert_{-\frac{1}{2},\Gamma_N} \, \Vert \myvec{u}_2-\myvec{u}_1\Vert_{\frac{1}{2},\Gamma_N} \\
 &\leq \left( \Vert \myvec{f}_2-\myvec{f}_1 \Vert_{V^*} + c_{tr} \, \Vert \myvec{g}_2-\myvec{g}_1 \Vert_{-\frac{1}{2},\Gamma_N} \right) \Vert \myvec{u}_2-\myvec{u}_1\Vert_{1,\Omega}.
\end{align*}
Hence, \eqref{BBS_eq:proofLEM4_vareq_03} yields
\begin{align*}
 \alpha \, \Vert (\myvec{u}_2-\myvec{u}_1, \mymatrix{p}_2-\mymatrix{p}_1) \Vert^2
 &\leq a\big( (\myvec{u}_2-\myvec{u}_1, \mymatrix{p}_2-\mymatrix{p}_1), (\myvec{u}_2-\myvec{u}_1, \mymatrix{p}_2-\mymatrix{p}_1) \big) \\
 &\leq \Vert \sigma_{y,2}-\sigma_{y,1}\Vert_{0,\Omega} \, \Vert\mymatrix{p}_2 - \mymatrix{p}_1\Vert_{0,\Omega}  + \Big( \Vert\myvec{f}_2-\myvec{f}_1\Vert_{V^*} + c_{tr} \, \Vert \myvec{g}_2 - \myvec{g}_1\Vert_{-\frac{1}{2},\Gamma_N} \Big) \Vert \myvec{u}_2-\myvec{u}_1 \Vert_{1,\Omega},
\end{align*}
from which we deduce
\begin{align}\label{BBS_eq:proofLEM4_vareq_04}
 \Vert (\myvec{u}_2-\myvec{u}_1, \mymatrix{p}_2-\mymatrix{p}_1) \Vert
 \leq \frac{1}{\alpha} \, \Big( \Vert \sigma_{y,2}-\sigma_{y,1} \Vert_{0,\Omega} + \Vert \myvec{f}_2-\myvec{f}_1 \Vert_{V^*} + c_{tr} \, \Vert \myvec{g}_2-\myvec{g}_1\Vert_{-\frac{1}{2},\Gamma_N} \Big).
\end{align}
Finally, the choice $\myvec{v}=\myvec{o}$ and $\mymatrix{q}=\mymatrix{\lambda}_2-\mymatrix{\lambda}_1$ as test functions in \eqref{BBS_eq:proofLEM4_vareq_02} gives
\begin{align*}
 \Vert \mymatrix{\lambda}_2-\mymatrix{\lambda}_1\Vert_{0,\Omega}^2
 = - a\big( (\myvec{u}_2-\myvec{u}_1, \mymatrix{p}_2 -\mymatrix{p}_1), (\myvec{0}, \mymatrix{\lambda}_2-\mymatrix{\lambda}_1) \big) 
 \leq c_a \, \Vert (\myvec{u}_2-\myvec{u}_1,\mymatrix{p}_2-\mymatrix{p}_1) \Vert \, \Vert \mymatrix{\lambda}_2- \mymatrix{\lambda}_1 \Vert_{0,\Omega}.
\end{align*}
This inequality together with \eqref{BBS_eq:proofLEM4_vareq_04} finally imply the assertion.
\end{proof}

\medskip


\section{An $hp$-Finite Element Discretization of the Mixed Formulation} 
\label{BBS_sec:discretization}

Let $\mathcal{T}_h$ be a locally quasi-uniform finite element mesh of $\Omega$ consisting of quadrilaterals or hexahedrons. Moreover, choose $\widehat{T} := [-1,1]^d$ as reference element and let $\myvec{F}_T:\widehat{T}\rightarrow T$ be the bi/tri-linear bijective mapping for $T\in\mathcal{T}_h$. We set $h := (h_T)_{T\in\mathcal{T}_h}$ and $p := (p_T)_{T\in\mathcal{T}_h}$ where $h_T$ and $p_T$ denote the local element size and the local polynomial degree, respectively. The polynomial degrees of neighboring elements are assumed to be comparable. We refer to \cite{Melenk2005} for details on quasi-uniformity and comparable polynomial degrees. As there is no risk of confusion we also use the standard notation $h:=\max_{T\in\mathcal{T}_h} h_T$ and $p:=\min_{T\in\mathcal{T}_h} p_T$. 

The a priori error analysis of the following sections frequently exploits the exactness of the Gauss quadrature for polynomials. For this purpose, we make the general assumption
\begin{align}\label{BBS_eq:condition_detF}
 \det \nabla\myvec{F}_T\in\myspace{P}_1 (\widehat{T})
\end{align} 
for those elements $T\in\mathcal{T}_h$ for which there holds $p_T \geq 2$ . Note that $\det \nabla\myvec{F}_T$ has no change of sign on $\widehat{T}$. While \eqref{BBS_eq:condition_detF} is no restriction for the two dimensional case it does slightly limit the shape of mesh elements in the case of $d=3$. For the discretization of the displacement field and the plastic strain we use the $hp$-finite element spaces
\begin{align*}
 V_{hp} 
 &:= \Big\lbrace \myvec{v}_{hp}\in V \; ; \; \myvec{v}_{hp \, | \, T}\circ \myvec{F}_T\in \big(\myspace{P}_{p_T}(\widehat{T})\big)^d \text{ for all } T\in\mathcal{T}_h\Big\rbrace, \\
 Q_{hp} 
 &:= \Big\lbrace \mymatrix{q}_{hp}\in Q \; ; \; \mymatrix{q}_{hp \, | \, T}\circ \myvec{F}_T\in \big(\myspace{P}_{p_T-1}(\widehat{T})\big)^{d\times d} \text{ for all } T\in\mathcal{T}_h\Big\rbrace.
\end{align*}
In addition we need a non-empty, convex and closed set $\Lambda_{hp}$ of admissible discrete Lagrange multipliers. To this end, let $\hat{\myvec{x}}_{k,T}\in \widehat{T}$ for $1\leq k \leq n_T$ be the tensor product Gauss quadrature points on $\widehat{T}$, where $n_T := p_T^d$ for $T\in\mathcal{T}_h$, and define
\begin{align*}
 \Lambda_{hp} 
 := \Big\lbrace \mymatrix{\mu}_{hp}\in Q_{hp} \; ; \; \big\vert \mymatrix{\mu}_{hp}(\myvec{F}_T(\hat{\myvec{x}}_{k,T})) \big\vert_F \leq \sigma_y \text{ for all } 1\leq k \leq n_T \text{ and } T \in \mathcal{T}_h \Big\rbrace.
\end{align*}
Obviously, $\Lambda_{hp}$ is a non-empty, convex and closed subset of $Q_{hp}$. As the constraint on the Frobenius norm of $\mymatrix{\mu}_{hp}$ is only enforced in the Gauss quadrature points we have that $\Lambda_{hp}\subseteq \Lambda$ if and only if $p_T=1$ for all $T\in\mathcal{T}_h$. We emphasize that $Q_{hp}$ and $\Lambda_{hp}$ are defined on the same mesh and have the same polynomial degree distribution. Moreover, we can use the same Gauss-Legendre-Lagrange based basis functions for both sets, which simplifies the implementation and may reduce the computational effort.
The discrete mixed formulation is to find a triple $(\myvec{u}_{hp},\mymatrix{p}_{hp},\mymatrix{\lambda}_{hp})\in V_{hp}\times Q_{hp}\times\Lambda_{hp}$ such that
\begin{subequations}\label{BBS_eq:discrete_mixed_variationalF}
\begin{alignat}{2}
 a\big( (\myvec{u}_{hp},\mymatrix{p}_{hp}),(\myvec{v}_{hp},\mymatrix{q}_{hp}) \big) + (\mymatrix{\lambda}_{hp}, \mymatrix{q}_{hp})_{0,\Omega} &= \ell(\myvec{v}_{hp}) & \qquad & \forall \, (\myvec{v}_{hp},\mymatrix{q}_{hp}) \in V_{hp}\times Q_{hp}, \label{BBS_eq:discrete_mixed_variationalF_01}\\
 (\mymatrix{\mu}_{hp}-\mymatrix{\lambda}_{hp}, \mymatrix{p}_{hp})_{0,\Omega}
 &\leq 0 & \quad & \forall \, \mymatrix{\mu}_{hp} \in \Lambda_{hp}. \label{BBS_eq:discrete_mixed_variationalF_02}
\end{alignat}
\end{subequations}

The following lemma states that the discrete inf-sup condition is uniformly fulfilled. In fact, it holds with an equality sign and an inf-sup constant of $1$.
\begin{lemma}\label{BBS_lem:discrete_inf_sup}
There holds
\begin{align}\label{BBS_eq:babusca_brezzi_cond}
 \sup_{\substack{(\myvec{v}_{hp},\mymatrix{q}_{hp})\in V_{hp}\times Q_{hp} \\ \Vert (\myvec{v}_{hp},\mymatrix{q}_{hp})\Vert \neq 0}} \frac{(\mymatrix{\mu}_{hp},\mymatrix{q}_{hp})_{0,\Omega}}{\Vert(\myvec{v}_{hp},\mymatrix{q}_{hp})\Vert}
  = \Vert \mymatrix{\mu}_{hp}\Vert_{0,\Omega} \qquad 
  \forall \, \mymatrix{\mu}_{hp}\in Q_{hp}.
\end{align}
\end{lemma}

\begin{proof}
The Cauchy-Schwarz inequality immediately yields for arbitrary $\mymatrix{\mu}_{hp}\in Q_{hp}$
\begin{align*}
 \sup_{\substack{(\myvec{v}_{hp},\mymatrix{q}_{hp})\in V_{hp}\times Q_{hp} \\ \Vert (\myvec{v}_{hp},\mymatrix{q}_{hp})\Vert \neq 0}} \frac{(\mymatrix{\mu}_{hp},\mymatrix{q}_{hp})_{0,\Omega}}{\Vert(\myvec{v}_{hp},\mymatrix{q}_{hp})\Vert}
 \leq \Vert \mymatrix{\mu}_{hp}\Vert_{0,\Omega}.
\end{align*}
Conversely, as $(\myvec{0},\mymatrix{\mu}_{hp})\in V_{hp}\times Q_{hp}$ it holds
\begin{align*}
 \sup_{\substack{(\myvec{v}_{hp},\mymatrix{q}_{hp})\in V_{hp}\times Q_{hp} \\ \Vert (\myvec{v}_{hp},\mymatrix{q}_{hp})\Vert \neq 0}} \frac{(\mymatrix{\mu}_{hp},\mymatrix{q}_{hp})_{0,\Omega}}{\Vert(\myvec{v}_{hp},\mymatrix{q}_{hp})\Vert}
 \geq \Vert \mymatrix{\mu}_{hp}\Vert_{0,\Omega},
\end{align*}
which shows \eqref{BBS_eq:babusca_brezzi_cond}.
\end{proof}

The following theorem shows the existence of a unique solution as well as the stability of the discretization.

\begin{theorem}\label{BBS_thm:discrete_mixedF_solvable}
There exists a unique solution to \eqref{BBS_eq:discrete_mixed_variationalF}. Moreover, it holds
\begin{subequations}
\begin{align}
 \Vert (\myvec{u}_{hp},\mymatrix{p}_{hp}) \Vert &\leq \frac{1}{\alpha} \, \Big( \Vert \myvec{f}\Vert_{V^*} + c_{tr} \, \Vert \myvec{g}\Vert_{-\frac{1}{2},\Gamma_N} \Big), \label{BBS_eq:stabilityIQ_up} \\
 \Vert \mymatrix{\lambda}_{hp}\Vert_{0,\Omega} &\leq \frac{c_a}{\alpha} \, \Big( \Vert \myvec{f}\Vert_{V^*} + c_{tr} \, \Vert \myvec{g}\Vert_{-\frac{1}{2},\Gamma_N} \Big). \label{BBS_eq:stabilityIQ_lambda}
\end{align}
\end{subequations}
\end{theorem}

\begin{proof}
The mixed formulation \eqref{BBS_eq:discrete_mixed_variationalF} is equivalent to the saddle point problem: Find a
triple $(\myvec{u}_{hp},\mymatrix{p}_{hp},\mymatrix{\lambda}_{hp})\in V_{hp}\times Q_{hp}\times\Lambda_{hp}$ such that
\begin{align*}
 \mathcal{L}(\myvec{u}_{hp},\mymatrix{p}_{hp},\mymatrix{\mu}_{hp})
 \leq \mathcal{L}(\myvec{u}_{hp},\mymatrix{p}_{hp},\mymatrix{\lambda}_{hp})
 \leq \mathcal{L}(\myvec{v}_{hp},\mymatrix{q}_{hp},\mymatrix{\lambda}_{hp}) \qquad 
 \forall \, (\myvec{v}_{hp},\mymatrix{q}_{hp},\mymatrix{\mu}_{hp})\in V_{hp}\times Q_{hp}\times \Lambda_{hp}
\end{align*}
with the Lagrangian $\mathcal{L}:V_{hp}\times Q_{hp}\times \Lambda_{hp}\rightarrow \myspace{R}$, which is given by
\begin{align*}
 \mathcal{L}(\myvec{v}_{hp},\mymatrix{q}_{hp},\mymatrix{\mu}_{hp}) 
 := \frac{1}{2} \, a\big( (\myvec{v}_{hp},\mymatrix{q}_{hp}), (\myvec{v}_{hp},\mymatrix{q}_{hp}) \big) - \ell(\myvec{v}_{hp}) + (\mymatrix{\mu}_{hp},\mymatrix{q}_{hp})_{0,\Omega},
\end{align*}
see e.g.~\cite{Ekeland_1976}. Since $a(\cdot,\cdot)$ is elliptic and continuous, we easily find that $\mathcal{L}(\myvec{v}_{hp},\mymatrix{q}_{hp},\mymatrix{\mu}_{hp}) $ is concave and upper semi-continuous in $\mymatrix{\mu}_{hp}$, and convex, lower semi-continuous and coercive in $(\myvec{v}_{hp}, \mymatrix{q}_{hp})$. As $\Lambda_{hp}$ is a non-empty, convex and closed subset of $Q$, \cite[Prop.~2.3, Ch.VI]{Ekeland_1976} guarantees the existence of the discrete solution. To show the uniqueness let $(\myvec{u}_{hp},\mymatrix{p}_{hp},\mymatrix{\lambda}_{hp})$ and $(\widetilde{\myvec{u}}_{hp},\widetilde{\mymatrix{p}}_{hp},\widetilde{\mymatrix{\lambda}}_{hp})$ be two solutions of \eqref{BBS_eq:discrete_mixed_variationalF}. Subtracting the two equations obtained from \eqref{BBS_eq:discrete_mixed_variationalF_01} with the test function $(\myvec{u}_{hp}-\widetilde{\myvec{u}}_{hp},\mymatrix{p}_{hp}-\widetilde{\mymatrix{p}}_{hp})$ from each other leads to
\begin{align*}
 a\big( (\myvec{u}_{hp}-\widetilde{\myvec{u}}_{hp},\mymatrix{p}_{hp}-\widetilde{\mymatrix{p}}_{hp}),(\myvec{u}_{hp}-\widetilde{\myvec{u}}_{hp},\mymatrix{p}_{hp}-\widetilde{\mymatrix{p}}_{hp}) \big) 
 = - (\mymatrix{\lambda}_{hp}-\widetilde{\mymatrix{\lambda}}_{hp},\mymatrix{p}_{hp}-\widetilde{\mymatrix{p}}_{hp})_{0,\Omega} 
 \leq 0,
\end{align*}
where the inequality follows by adding the two inequalities \eqref{BBS_eq:discrete_mixed_variationalF_02} with the test function $\widetilde{\mymatrix{\lambda}}_{hp}$, $\mymatrix{\lambda}_{hp}$, respectively. The ellipticity of $a(\cdot,\cdot)$ now implies uniqueness of the primal variable $(\myvec{u}_{hp},\mymatrix{p}_{hp})$. With that at hand, subtracting again the two equations \eqref{BBS_eq:discrete_mixed_variationalF_01} from each other with the test function $(\myvec{0},\mymatrix{\lambda}_{hp}-\widetilde{\mymatrix{\lambda}}_{hp})$ yields
\begin{align*}
 (\mymatrix{\lambda}_{hp}-\widetilde{\mymatrix{\lambda}}_{hp} ,\mymatrix{\lambda}_{hp}-\widetilde{\mymatrix{\lambda}}_{hp} )_{0,\Omega}
 = 0.
\end{align*}
This implies the remaining uniqueness of the dual variable $\mymatrix{\lambda}_{hp}$ (which can be also deduced from the discrete inf-sup condition \eqref{BBS_eq:babusca_brezzi_cond}). \\

To show the stability estimates we choose the test function $(\myvec{u}_{hp},\mymatrix{p}_{hp},\mymatrix{0})$ in \eqref{BBS_eq:discrete_mixed_variationalF} and obtain
\begin{align*}
 \alpha \, \Vert (\myvec{u}_{hp},\mymatrix{p}_{hp}) \Vert^2 
 &\leq a\big( (\myvec{u}_{hp},\mymatrix{p}_{hp}),(\myvec{u}_{hp},\mymatrix{p}_{hp}) \big) \\
 &= \ell(\myvec{u}_{hp}) - (\mymatrix{\lambda}_{hp}, \mymatrix{p}_{hp})_{0,\Omega} \\
 &\leq \ell(\myvec{u}_{hp}) \\
 &\leq \Big( \Vert \myvec{f}\Vert_{V^*} + c_{tr} \, \Vert \myvec{g}\Vert_{-\frac{1}{2},\Gamma_N} \Big) \Vert \myvec{u}_{hp}\Vert_{1,\Omega}.
\end{align*}
Hence, we deduce
\begin{align*}
 \Vert (\myvec{u}_{hp},\mymatrix{p}_{hp}) \Vert 
 \leq \frac{1}{\alpha} \, \Big( \Vert \myvec{f}\Vert_{V^*} + c_{tr} \, \Vert \myvec{g}\Vert_{-\frac{1}{2},\Gamma_N} \Big),
\end{align*}
which is \eqref{BBS_eq:stabilityIQ_up}. Next, the discrete inf-sup condition \eqref{BBS_eq:babusca_brezzi_cond} together with \eqref{BBS_eq:discrete_mixed_variationalF_01} gives
\begin{align*}
 \Vert \mymatrix{\lambda}_{hp}\Vert_{0,\Omega}
 &= \sup_{\mymatrix{q}_{hp}\in Q_{hp}\setminus\lbrace \mymatrix{0}\rbrace} \frac{(\mymatrix{\lambda}_{hp},\mymatrix{q}_{hp})_{0,\Omega}}{\Vert \mymatrix{q}_{hp}\Vert_{0,\Omega}} 
 = \sup_{\mymatrix{q}_{hp}\in Q_{hp}\setminus\lbrace \mymatrix{0}\rbrace} \frac{ - a\big( (\myvec{u}_{hp},\mymatrix{p}_{hp}),(\myvec{o},\mymatrix{q}_{hp}) \big)}{\Vert \mymatrix{q}_{hp}\Vert_{0,\Omega}} 
 \leq c_a \, \Vert (\myvec{u}_{hp}, \mymatrix{p}_{hp})\Vert.
\end{align*}
Combining the last two estimates finally proves \eqref{BBS_eq:stabilityIQ_lambda}.
\end{proof}

The discrete mixed problem \eqref{BBS_eq:discrete_mixed_variationalF} may be solved with an Uzawa algorithm as \eqref{BBS_eq:discrete_mixed_variationalF_02} implies
\begin{align*}
 \mymatrix{\lambda}_{hp} 
 = \Pi_{\Lambda_{hp}} \left(\mymatrix{\lambda}_{hp} + \rho \, \mymatrix{p}_{hp} \right) \qquad 
 \forall \, \rho >0,
\end{align*}
where $\Pi_{\Lambda_{hp}}$ represents the $L^2$-projection on $\Lambda_{hp}$. Alternatively, one may use the semi-smooth Newton method, introduced in \cite{Bammer2022Icosahom}, which has super-linear convergence properties.

\medskip


\section{A Priori Error Analysis}\label{BBS_sec:a_priori_analysis}

The following a priori error estimate (Theorem~\ref{BBS_prop:apriori_estimate_hild}) is the foundation of the convergence analysis presented in this section. It mainly follows the ideas in \cite{hild2002quadratic}, where unilateral contact problems and a discretization with quadratic finite elements are considered.

\begin{theorem}\label{BBS_prop:apriori_estimate_hild}
There exist two positive constants $c_1$ and $c_2$ such that 
\begin{align*}
   \Vert (\myvec{u}-\myvec{u}_{hp}, \mymatrix{p}-\mymatrix{p}_{hp}) \Vert^2  + \Vert\mymatrix{\lambda}-\mymatrix{\lambda}_{hp}\Vert_{0,\Omega}^2
 & \leq  c_1 \, \Big( \Vert (\myvec{u}-\myvec{v}_{hp}, \mymatrix{p}-\mymatrix{q}_{hp}) \Vert^2 +  \Vert \mymatrix{\lambda}-\mymatrix{\mu}_{hp}\Vert_{0,\Omega}^2 \Big)  \\
 & \qquad + c_2 \, (\mymatrix{\lambda}_{hp}-\mymatrix{\mu} + \mymatrix{\lambda}-\mymatrix{\mu}_{hp}, \mymatrix{p})_{0,\Omega} 
\end{align*}
for all $(\myvec{v}_{hp},\mymatrix{q}_{hp},\mymatrix{\mu}_{hp})\in V_{hp}\times Q_{hp}\times \Lambda_{hp}$ and all $\mymatrix{\mu}\in\Lambda$.
\end{theorem}
 
\begin{proof}
Subtracting \eqref{BBS_eq:discrete_mixed_variationalF_01} from \eqref{BBS_eq:mixed_variationalF_01} gives
\begin{align}\label{BBS_eq:proof_aPriori_hild_01}
 a\big( (\myvec{u}-\myvec{u}_{hp},\mymatrix{p}-\mymatrix{p}_{hp}), (\myvec{v}_{hp},\mymatrix{q}_{hp}) \big) + (\mymatrix{\lambda}-\mymatrix{\lambda}_{hp}, \mymatrix{q}_{hp})_{0,\Omega} 
 = 0.
\end{align}
Thus, $\mymatrix{\mu}_{hp} - \mymatrix{\lambda}_{hp}\in Q_{hp}$, the Cauchy-Schwarz inequality and the continuity of $a(\cdot,\cdot)$ yield
\begin{align*}
 \Vert \mymatrix{\mu}_{hp}-\mymatrix{\lambda}_{hp}\Vert_{0,\Omega} 
 &= \sup_{\mymatrix{q}_{hp}\in Q_{hp}\setminus\lbrace \mymatrix{0} \rbrace} \frac{(\mymatrix{\mu}_{hp}-\mymatrix{\lambda}_{hp}, \mymatrix{q}_{hp})_{0,\Omega}}{\Vert \mymatrix{q}_{hp}\Vert_{0,\Omega}} \\
 &=\sup_{\mymatrix{q}_{hp}\in Q_{hp}\setminus\lbrace \mymatrix{0} \rbrace} \frac{(\mymatrix{\mu}_{hp}-\mymatrix{\lambda}, \mymatrix{q}_{hp})_{0,\Omega} + (\mymatrix{\lambda}-\mymatrix{\lambda}_{hp}, \mymatrix{q}_{hp})_{0,\Omega}}{\Vert \mymatrix{q}_{hp}\Vert_{0,\Omega}} \\
 &=\sup_{\mymatrix{q}_{hp}\in Q_{hp}\setminus\lbrace \mymatrix{0} \rbrace} \frac{(\mymatrix{\mu}_{hp}-\mymatrix{\lambda}, \mymatrix{q}_{hp})_{0,\Omega} - a\big( (\myvec{u}-\myvec{u}_{hp},\mymatrix{p}-\mymatrix{p}_{hp}), (\myvec{o},\mymatrix{q}_{hp}) \big)}{\Vert \mymatrix{q}_{hp}\Vert_{0,\Omega}} \\
 &\leq \Vert \mymatrix{\lambda}-\mymatrix{\mu}_{hp}\Vert_{0,\Omega} + c_a \, \Vert (\myvec{u}-\myvec{u}_{hp}, \mymatrix{p}-\mymatrix{p}_{hp}) \Vert.
\end{align*}
Now, the triangle inequality implies
\begin{align}\label{eq:error_lambda_est_part1}
 \Vert \mymatrix{\lambda}-\mymatrix{\lambda}_{hp}\Vert_{0,\Omega}
 \leq \Vert \mymatrix{\lambda}-\mymatrix{\mu}_{hp}\Vert_{0,\Omega} + \Vert \mymatrix{\mu}_{hp}-\mymatrix{\lambda}_{hp}\Vert_{0,\Omega}
 \leq 2 \, \Vert \mymatrix{\lambda}-\mymatrix{\mu}_{hp}\Vert_{0,\Omega} + c_a \, \Vert (\myvec{u}-\myvec{u}_{hp}, \mymatrix{p}-\mymatrix{p}_{hp}) \Vert,
\end{align}
from which we deduce
\begin{align}\label{BBS_eq:proof_aPriori_hild_03}
 \Vert (\myvec{u}-\myvec{u}_{hp}, \mymatrix{p}-\mymatrix{p}_{hp}) \Vert^2 + \Vert \mymatrix{\lambda}-\mymatrix{\lambda}_{hp}\Vert_{0,\Omega}^2 
 \leq  (1 + 2 c_a^2) \,\Vert (\myvec{u}-\myvec{u}_{hp}, \mymatrix{p}-\mymatrix{p}_{hp}) \Vert^2 + 8 \, \Vert \mymatrix{\lambda}-\mymatrix{\mu}_{hp}\Vert_{0,\Omega}^2.
\end{align} 
From \eqref{BBS_eq:mixed_variationalF_02}, \eqref{BBS_eq:discrete_mixed_variationalF_02} and Young's inequality we obtain for any $\varepsilon>0$, $\mymatrix{\mu} \in \Lambda$ and $\mymatrix{\mu}_{hp} \in \Lambda_{hp}$ that
\begin{align*}
 (\mymatrix{\lambda}_{hp}-\mymatrix{\lambda}, \mymatrix{p}-\mymatrix{p}_{hp})_{0,\Omega}
 &\leq (\mymatrix{\lambda}_{hp}-\mymatrix{\lambda}, \mymatrix{p}-\mymatrix{p}_{hp})_{0,\Omega} + (\mymatrix{\lambda}-\mymatrix{\mu},\mymatrix{p})_{0,\Omega} + (\mymatrix{\lambda}_{hp}-\mymatrix{\mu}_{hp}, \mymatrix{p}_{hp})_{0,\Omega} \\
 &= (\mymatrix{\lambda}_{hp}-\mymatrix{\mu}, \mymatrix{p})_{0,\Omega} + (\mymatrix{\lambda}-\mymatrix{\mu}_{hp}, \mymatrix{p}_{hp})_{0,\Omega} \\
 &= (\mymatrix{\lambda}_{hp}-\mymatrix{\mu}, \mymatrix{p})_{0,\Omega} + (\mymatrix{\lambda}-\mymatrix{\mu}_{hp}, \mymatrix{p})_{0,\Omega} + (\mymatrix{\lambda}-\mymatrix{\mu}_{hp}, \mymatrix{p}_{hp}-\mymatrix{p})_{0,\Omega} \\
 &\leq (\mymatrix{\lambda}_{hp}-\mymatrix{\mu} + \mymatrix{\lambda}-\mymatrix{\mu}_{hp}, \mymatrix{p})_{0,\Omega} + \varepsilon \, \Vert \mymatrix{p}-\mymatrix{p}_{hp}\Vert_{0,\Omega}^2 + \frac{1}{4\varepsilon} \, \Vert \mymatrix{\lambda}-\mymatrix{\mu}_{hp}\Vert_{0,\Omega}^2.
\end{align*}
Hence, by using \eqref{BBS_eq:proof_aPriori_hild_01} and applying Young's inequality we obtain
\begin{align*}
 \alpha \, \Vert (\myvec{u}-\myvec{u}_{hp}, \mymatrix{p}-\mymatrix{p}_{hp}) \Vert^2 
 &\leq a\big( (\myvec{u}-\myvec{u}_{hp}, \mymatrix{p}-\mymatrix{p}_{hp}), (\myvec{u}-\myvec{u}_{hp}, \mymatrix{p}-\mymatrix{p}_{hp}) \big) \\
 &= a\big( (\myvec{u}-\myvec{u}_{hp}, \mymatrix{p}-\mymatrix{p}_{hp}), (\myvec{u}-\myvec{v}_{hp}, \mymatrix{p}-\mymatrix{q}_{hp}) \big) - (\mymatrix{\lambda}-\mymatrix{\lambda}_{hp}, \mymatrix{q}_{hp}-\mymatrix{p}_{hp})_{0,\Omega} \\
 &= a\big( (\myvec{u}-\myvec{u}_{hp}, \mymatrix{p}-\mymatrix{p}_{hp}), (\myvec{u}-\myvec{v}_{hp}, \mymatrix{p}-\mymatrix{q}_{hp}) \big) + (\mymatrix{\lambda}-\mymatrix{\lambda}_{hp}, \mymatrix{p}-\mymatrix{q}_{hp})_{0,\Omega}\\ 
 & \qquad  - (\mymatrix{\lambda}-\mymatrix{\lambda}_{hp}, \mymatrix{p}-\mymatrix{p}_{hp})_{0,\Omega} \\
 &\leq c_a \, \Vert (\myvec{u}-\myvec{u}_{hp}, \mymatrix{p}-\mymatrix{p}_{hp}) \Vert \, \Vert (\myvec{u}-\myvec{v}_{hp}, \mymatrix{p}-\mymatrix{q}_{hp}) \Vert + \Vert \mymatrix{\lambda}-\mymatrix{\lambda}_{hp}\Vert_{0,\Omega} \, \Vert \mymatrix{p}-\mymatrix{q}_{hp}\Vert_{0,\Omega}\\
 & \qquad  + (\mymatrix{\lambda}_{hp}-\mymatrix{\lambda}, \mymatrix{p}-\mymatrix{p}_{hp})_{0,\Omega} \\
 & \leq \varepsilon \, \Vert (\myvec{u}-\myvec{u}_{hp}, \mymatrix{p}-\mymatrix{p}_{hp}) \Vert^2 + \frac{c_a^2}{4\varepsilon} \, \Vert (\myvec{u}-\myvec{v}_{hp}, \mymatrix{p}-\mymatrix{q}_{hp}) \Vert^2 + \varepsilon \, \Vert \mymatrix{\lambda}-\mymatrix{\lambda}_{hp}\Vert_{0,\Omega}^2 \\
 & \qquad + \frac{1}{4\varepsilon} \, \Vert \mymatrix{p}-\mymatrix{q}_{hp}\Vert_{0,\Omega}^2+(\mymatrix{\lambda}_{hp}-\mymatrix{\mu} + \mymatrix{\lambda}-\mymatrix{\mu}_{hp}, \mymatrix{p})_{0,\Omega} + \varepsilon \, \Vert \mymatrix{p}-\mymatrix{p}_{hp}\Vert_{0,\Omega}^2 \\ 
 & \qquad \qquad + \frac{1}{4\varepsilon} \, \Vert \mymatrix{\lambda}-\mymatrix{\mu}_{hp}\Vert_{0,\Omega}^2,
\end{align*}
from which we deduce, by inserting \eqref{eq:error_lambda_est_part1}, that
\begin{align*}
 \big(\alpha-2(1+c_a^2)\varepsilon\big) \, \Vert (\myvec{u}-\myvec{u}_{hp}, \mymatrix{p}-\mymatrix{p}_{hp}) \Vert^2 
 & \leq  \frac{1+c_a^2}{4\varepsilon} \, \Vert (\myvec{u}-\myvec{v}_{hp}, \mymatrix{p}-\mymatrix{q}_{hp}) \Vert^2 + \left(8\varepsilon + \frac{1}{4\varepsilon} \right) \, \Vert \mymatrix{\lambda}-\mymatrix{\mu}_{hp}\Vert_{0,\Omega}^2  \\
 & \qquad +(\mymatrix{\lambda}_{hp}-\mymatrix{\mu} + \mymatrix{\lambda}-\mymatrix{\mu}_{hp}, \mymatrix{p})_{0,\Omega} .
\end{align*}
For any $\varepsilon < \alpha \, (2 + 2c_a^2)^{-1}$ we therefore obtain
\begin{align*}
 \Vert (\myvec{u}-\myvec{u}_{hp}, \mymatrix{p}-\mymatrix{p}_{hp}) \Vert^2 
 & \leq \frac{1 + c_a^2}{4\varepsilon c_\varepsilon} \, \Vert (\myvec{u}-\myvec{v}_{hp}, \mymatrix{p}-\mymatrix{q}_{hp}) \Vert^2 + \frac{1 + 32\varepsilon^2}{4\varepsilon c_\varepsilon} \, \Vert \mymatrix{\lambda}-\mymatrix{\mu}_{hp}\Vert_{0,\Omega}^2  \\
 & \qquad + \frac{1}{c_\varepsilon} \, (\mymatrix{\lambda}_{hp}-\mymatrix{\mu} + \mymatrix{\lambda}-\mymatrix{\mu}_{hp}, \mymatrix{p})_{0,\Omega}
\end{align*}
with $c_\varepsilon := \alpha - 2\varepsilon(1 + c_a^2)$, which together with \eqref{BBS_eq:proof_aPriori_hild_03} yields the assertion with the constants
\begin{align*}
 c_1 := \max \left\lbrace (1+2c_a^2) \, \frac{1 + c_a^2}{4\varepsilon c_\varepsilon}, \; (1+2c_a^2) \left( \frac{1 + 32\varepsilon^2}{4\varepsilon c_\varepsilon} + 8 \right) \right\rbrace, \qquad
 c_2 := \frac{1 + 2c_a^2}{c_\varepsilon}.
\end{align*}
\end{proof}

In the following, we make use of an equivalent representation of the set $\Lambda_{hp}$ of discrete Lagrange multipliers. For this purpose, let $\hat{\omega}_{k,T}\in\myspace{R}$ be the positive weight associated with the tensor product Gauss quadrature point $\hat{\myvec{x}}_{k,T}\in\widehat{T}$ for $1\leq k\leq n_T$ and $T\in\mathcal{T}_h$. Moreover, let $\{\hat{\phi}_{k,T}\}_{k=1,\ldots,n_T}$ be the Lagrange basis functions on $\widehat{T}$ defined via these Gauss points, i.e. 
\begin{align*}
 \hat{\phi}_{k,T}\in\myspace{P}_{p_T-1}(\widehat{T}),
 \quad 
 \hat{\phi}_{k,T}(\hat{\myvec{x}}_{l,T}) = \delta_{kl} \qquad 
 \forall \, 1\leq k,l\leq n_T \quad \forall \, T\in\mathcal{T}_h,
\end{align*}
where $\delta_{kl}$ is the usual Kronecker delta symbol. These basis functions allow to define the nodal interpolation operator $\mathcal{J}_{hp}(\cdot)$ element wise by
\begin{align}\label{BBS_eq:interpolOperator_Jhp}
 \mathcal{J}_{hp}(q)_{\, | \, T} 
 := \sum_{k=1}^{n_T} q(\myvec{F}_T(\hat{\myvec{x}}_{k,T}))\, \hat{\phi}_{k,T}\circ\myvec{F}_T^{-1} \qquad 
 \forall \, T\in\mathcal{T}_h
\end{align}
for $q\in L^2(\Omega)$ with $q_{\, | \, T}\in C^0(T)$ for $T\in\mathcal{T}_h$. Thereby, we define the discretized plasticity functional $\psi_{hp}: Q_{hp}\rightarrow\myspace{R}$ as
\begin{align} \label{eq:def_psi_hp}
 \psi_{hp}(\mymatrix{q}_{hp}) 
 &:= \int_{\Omega} \sigma_y \, \mathcal{J}_{hp}(\vert \mymatrix{q}_{hp}\vert_F) \mathrm{\, d}\myvec{x}.
\end{align}
To obtain a directly implementable representation of $\psi_{hp}(\cdot)$ we first note that
\begin{align*}
 \psi_{hp}(\mymatrix{q}_{hp}) 
 & = \sum_{T\in\mathcal{T}_h} \int_T \sigma_y \, \mathcal{J}_{hp}(\vert \mymatrix{q}_{hp}\vert_F) \mathrm{\, d}\myvec{x}
 = \sum_{T\in\mathcal{T}_h} \sum_{k=1}^{n_T} \sigma_y \, \big\vert \mymatrix{q}_{hp}\big(\myvec{F}_T(\hat{\myvec{x}}_{k,T})\big)\big\vert_F \int_T \hat{\phi}_{k,T}\circ\myvec{F}_T^{-1}(\myvec{x}) \mathrm{\, d}\myvec{x}.
\end{align*}
If $p_T = 1$ there holds
\begin{align*}
 \int_T \hat{\phi}_{1,T}\circ\myvec{F}_T^{-1}(\myvec{x}) \mathrm{\, d}\myvec{x} 
 = |T|
\end{align*}
as $\hat{\phi}_{1,T}\equiv 1$. If $p_T\geq 2$ the exactness of the Gauss quadrature yields
\begin{align*}
 \int_T \hat{\phi}_{k,T}\circ\myvec{F}_T^{-1}(\myvec{x}) \mathrm{\, d}\myvec{x} 
 &= \int_{\widehat{T}} \hat{\phi}_{k,T}(\hat{\myvec{x}}) \, \vert \det \nabla \myvec{F}_T(\hat{\myvec{x}})\vert \mathrm{\, d}\hat{\myvec{x}} \\
 &= \sum_{l=1}^{n_T} \hat{\omega}_{l,T} \,  \hat{\phi}_{k,T}(\hat{\myvec{x}}_{l,T}) \, \vert \det \nabla \myvec{F}_T(\hat{\myvec{x}}_{l,T})\vert \\
 &= \hat{\omega}_{k,T} \, \hat{\phi}_{k,T}(\hat{\myvec{x}}_{k,T}) \, \vert \det \nabla \myvec{F}_T(\hat{\myvec{x}}_{k,T})\vert
\end{align*}
for $1\leq k\leq n_T$. Thus,
\begin{align} \label{eq:psi_ha_as_quadrature}
 \psi_{hp}(\mymatrix{q}_{hp}) 
 = \mathcal{Q}_{hp}\big( \sigma_y \, \vert \mymatrix{q}_{hp} \vert_F \big),
\end{align}
where the mesh dependent quadrature rule $\mathcal{Q}_{hp}(\cdot)$ is given by $\mathcal{Q}_{hp}(\cdot) := \sum_{T\in\mathcal{T}_h} \mathcal{Q}_{hp, T}(\cdot)$
with the local quantities
\begin{align*}
 \mathcal{Q}_{hp, T}(f) 
 := \begin{cases}
 |T| \, f\big(\myvec{F}_T(\myvec{0})\big), & \text{if } p_T = 1, \\
 \sum_{k=1}^{n_T} \hat{\omega}_{k,T} \, \vert \det \nabla \myvec{F}_T(\hat{\myvec{x}}_{k,T})\vert \,  f\big(\myvec{F}_T(\hat{\myvec{x}}_{k,T})\big), & \text{if } p_T\geq 2,
 \end{cases}
 \qquad T\in\mathcal{T}_h.
\end{align*}
Note that on elements $T\in\mathcal{T}_h$ with $p_T \geq 2$ the expression $\mathcal{Q}_{hp}(\cdot)$ represents the standard Gauss quadrature on the reference element while for those elements with $p_T=1$ it is the midpoint rule on the element $T\in\mathcal{T}_h$ since the Gauss quadrature on the reference element with only one point is not exact even for constant $f$ in the case that $|\det \nabla\myvec{F}_T|$ is a polynomial of degree $\geq 2$.

\begin{theorem}\label{BBS_prop:special_cases}
There holds
\begin{align}\label{BBS_eq:characterization_Lambda_hp}
 \Lambda_{hp} 
 = \big\lbrace \mymatrix{\mu}_{hp}\in Q_{hp} \; ; \; (\mymatrix{\mu}_{hp},\mymatrix{q}_{hp})_{0,\Omega} \leq \psi_{hp}(\mymatrix{q}_{hp}) \text{ for all } \mymatrix{q}_{hp}\in Q_{hp}  \big\rbrace.
\end{align}
\end{theorem}

\begin{proof}
Let $\mymatrix{\mu}_{hp} \in \Lambda_{hp}$ and $\mymatrix{q}_{hp} \in Q_{hp}$ be arbitrary. If $p_T=1$, then $\mymatrix{\mu}_{hp \, | \, T}$ and $\mymatrix{q}_{hp \, | \, T}$ are constant on the element $T\in\mathcal{T}_h$ with $\big|\mymatrix{\mu}_{hp \, | \, T} \big|_F \leq \sigma_y$. Thus,
\begin{align*}
 \int_T \mymatrix{\mu}_{hp}:\mymatrix{q}_{hp}\mathrm{\, d}\myvec{x} 
 \leq \int_T |\mymatrix{\mu}_{hp}|_F \, |\mymatrix{q}_{hp}|_F\mathrm{\, d}\myvec{x}  
 \leq \int_T \sigma_y \, |\mymatrix{q}_{hp}|_F\mathrm{\, d}\myvec{x} 
 = \int_T \sigma_y \, \mathcal{J}_{hp}(\vert \mymatrix{q}_{hp}\vert_F)\mathrm{\, d} \myvec{x} 
 = \mathcal{Q}_{hp, T} \big( \sigma_y \, \vert \mymatrix{q}_{hp} \vert_F \big).
\end{align*}
If $p_T \geq 2$ on $T\in\mathcal{T}_h$ we have $|\det \nabla \myvec{F}_T|\in\myspace{P}_1(\widehat{T})$ by \eqref{BBS_eq:condition_detF}. Hence, we exploit the exactness of the Gauss quadrature for polynomials of degree $2p_T - 1$ and the positivity of the corresponding weights to obtain
\begin{align*}
 \int_T \mymatrix{\mu}_{hp}:\mymatrix{q}_{hp}\mathrm{\, d}\myvec{x}
 &=  \int_{\widehat{T}} \mymatrix{\mu}_{hp}\big(\myvec{F}_T(\hat{\myvec{x}})\big) : \mymatrix{q}_{hp}\big(\myvec{F}_T(\hat{\myvec{x}})\big) \, \big|\det \nabla \myvec{F}_T(\hat{\myvec{x}}) \big| \mathrm{\, d}\hat{\myvec{x}} \\
 &=  \sum_{k=1}^{n_T} \hat{\omega}_{k,T} \, \mymatrix{\mu}_{hp}\big(\myvec{F}_T(\hat{\myvec{x}}_{k,T})\big) : \mymatrix{q}_{hp}\big(\myvec{F}_T(\hat{\myvec{x}}_{k,T})\big) \, \big|\det \nabla \myvec{F}_T(\hat{\myvec{x}}_{k,T}) \big| \\
 &\leq  \sum_{k=1}^{n_T} \hat{\omega}_{k,T} \, \big\vert \mymatrix{\mu}_{hp}\big(\myvec{F}_T(\hat{\myvec{x}}_{k,T})\big) \big\vert_F \, \big\vert \mymatrix{q}_{hp}\big(\myvec{F}_T(\hat{\myvec{x}}_{k,T})\big) \big\vert_F \, \big|\det \nabla \myvec{F}_T(\hat{\myvec{x}}_{k,T}) \big| \\
 &\leq   \sum_{k=1}^{n_T} \sigma_y \, \big\vert \mymatrix{q}_{hp}\big(\myvec{F}_T(\hat{\myvec{x}}_{k,T})\big) \big\vert_F \, \hat{\omega}_{k,T} \, \big|\det \nabla \myvec{F}_T(\hat{\myvec{x}}_{k,T}) \big| \\
 &=   \sum_{k=1}^{n_T} \sigma_y \, \big\vert \mymatrix{q}_{hp}\big(\myvec{F}_T(\hat{\myvec{x}}_{k,T})\big) \big\vert_F \sum_{l=1}^{n_T} \hat{\omega}_{l,T} \, \hat{\phi}_{k,T}(\hat{\myvec{x}}_{l,T}) \, \big|\det \nabla \myvec{F}_T(\hat{\myvec{x}}_{k,T}) \big| \\
 &=  \sum_{k=1}^{n_T} \sigma_y \, \big\vert \mymatrix{q}_{hp}\big(\myvec{F}_T(\hat{\myvec{x}}_{k,T})\big) \big\vert_F \int_{\widehat{T}} \hat{\phi}_{k,T}(\hat{\myvec{x}}) \, \big|\det \nabla \myvec{F}_T(\hat{\myvec{x}}) \big| \mathrm{\, d}\hat{\myvec{x}} \\
 &=  \sum_{k=1}^{n_T} \sigma_y \, \big\vert \mymatrix{q}_{hp}\big(\myvec{F}_T(\hat{\myvec{x}}_{k,T})\big) \big\vert_F \int_{T} \hat{\phi}_{k,T}\big(\myvec{F}_T^{-1}(\myvec{x})\big) \mathrm{\, d} \myvec{x} \\
 &=   \int_T \sigma_y  \sum_{k=1}^{n_T} \big\vert \mymatrix{q}_{hp}\big(\myvec{F}_T(\hat{\myvec{x}}_{k,T})\big) \big\vert_F \, \hat{\phi}_{k,T}\big(\myvec{F}_T^{-1}(\myvec{x})\big)  \mathrm{\, d} \myvec{x} \\
 & = \int_T \sigma_y \, \mathcal{J}_{hp}(\vert \mymatrix{q}_{hp}\vert_F)\mathrm{\, d} \myvec{x} \\
 &= \mathcal{Q}_{hp, T} \big( \sigma_y \, \vert \mymatrix{q}_{hp} \vert_F \big).
\end{align*}
Summing up over all elements yields $(\mymatrix{\mu}_{hp},\mymatrix{q}_{hp})_{0,\Omega} \leq \mathcal{Q}_{hp} \big( \sigma_y \, \vert \mymatrix{q}_{hp} \vert_F \big) = \psi_{hp}(\mymatrix{q}_{hp})$.\\

Conversely, if $\mymatrix{\mu}_{hp}\in Q_{hp}$ satisfies $(\mymatrix{\mu}_{hp}, \mymatrix{q}_{hp})_{0,\Omega} \leq \psi_{hp}(\mymatrix{q}_{hp})$ for all $\mymatrix{q}_{hp}\in Q_{hp}$ it is also true for
\begin{align*}
 \widetilde{\mymatrix{q}}_{hp} := \begin{cases}
    \mymatrix{\mu}_{hp}(\myvec{F}_T(\hat{\myvec{x}}_{l,T})) \, \phi_{l,T} , &\text{ on }  T, \\
    \mymatrix{0} ,  &\text{ on } \Omega\setminus T
\end{cases}
\end{align*}
for any $1\leq l\leq n_T$ and $T\in\mathcal{T}_h$. If $p_T = 1$ we have 
\begin{align*}
 |T| \, \big|\mymatrix{\mu}_{hp \, | \, T} \big|_F^2 
 = \int_T |\mymatrix{\mu}_{hp}|_F^2 \mathrm{\, d}\myvec{x} 
 = (\mymatrix{\mu}_{hp}, \widetilde{\mymatrix{q}}_{hp})_{0,\Omega} 
 \leq \mathcal{Q}_{hp, T} \big( \sigma_y \, \vert \mymatrix{\mu}_{hp \, | \, T} \vert_F \big)
 = |T| \, \sigma_y \, \big|\mymatrix{\mu}_{hp \, | \, T} \big|_F ,
\end{align*}
which implies $|\mymatrix{\mu}_{hp \, | \, T}|_F \leq \sigma_y$. If $p_T \geq 2$ on $T\in\mathcal{T}_h$ the assumption $|\det \nabla \myvec{F}_T|\in\myspace{P}_1(\widehat{T})$, the exactness of the Gauss quadrature and the support of $\widetilde{\mymatrix{q}}_{hp}$ yield
\begin{align*}
 \hat{\omega}_{l,T} \, \big\vert \mymatrix{\mu}_{hp}\big(\myvec{F}_T(\hat{\myvec{x}}_{l,T})\big) \big\vert_F^2 \, \big|\det \nabla \myvec{F}_T(\hat{\myvec{x}}_{l,T}) \big| 
 &= \sum_{k=1}^{n_T} \hat{\omega}_{k,T} \, \mymatrix{\mu}_{hp}\big(\myvec{F}_T(\hat{\myvec{x}}_{k,T})\big) : \big(\mymatrix{\mu}_{hp}\big(\myvec{F}_T(\hat{\myvec{x}}_{l,T})\big) \, \hat{\phi}_{l,T}(\hat{\myvec{x}}_{k,T})\big) \, \big|\det \nabla \myvec{F}_T(\hat{\myvec{x}}_{k,T}) \big| \\
 &= (\mymatrix{\mu}_{hp}, \widetilde{\mymatrix{q}}_{hp})_{0,\Omega} \\
 & \leq \mathcal{Q}_{hp} \big( \sigma_y \, \vert \widetilde{\mymatrix{q}}_{hp} \vert_F \big) \\
 &= \sum_{k=1}^{n_T} \hat{\omega}_{k,T} \, \big|\det \nabla \myvec{F}_T(\hat{\myvec{x}}_{k,T}) \big| \, \sigma_y \, \big\vert \mymatrix{\mu}_{hp}\big(\myvec{F}_T(\hat{\myvec{x}}_{l,T})\big) \big\vert_F \, \hat{\phi}_{l,T}(\hat{\myvec{x}}_{k,T}) \\
 &= \sigma_y \, \hat{\omega}_{l,T} \, \big\vert \mymatrix{\mu}_{hp}\big(\myvec{F}_T(\hat{\myvec{x}}_{l,T})\big) \big\vert_F \, \big|\det \nabla \myvec{F}_T(\hat{\myvec{x}}_{l,T}) \big|.
\end{align*}
Since $\hat{\omega}_{l,T} \big|\det \nabla \myvec{F}_T(\hat{\myvec{x}}_{l,T}) \big|>0$ we obtain $\big\vert \mymatrix{\mu}_{hp}\big(\myvec{F}_T(\hat{\myvec{x}}_{l,T})\big) \big\vert_F \leq \sigma_y$ and, thus, $\mymatrix{\mu}_{hp} \in \Lambda_{hp}$.
\end{proof}

In the following we use the symbol $\lesssim$ to hide a constant $c>0$ in the expression $A \leq c \, B$ which is independent of $h$ and $p$. To derive guaranteed convergence rates we apply two projection operators: An $H^1$-projection operator $\mathcal{I}_{hp} : H^1(\Omega,\myspace{R}^d)\rightarrow V_{hp}$ with the approximation property 
\begin{align} \label{BBS_eq:interpolError_Ihp}
 \Vert \myvec{v}-\mathcal{I}_{hp}(\myvec{v})\Vert_{1,\Omega}
 \lesssim \frac{h^{\min(p,s-1)}}{p^{s-1}} \, |\myvec{v}|_{s,\Omega} \qquad 
 \forall \,  \myvec{v} \in H^s(\Omega,\myspace{R}^d) \text{ with } s\geq 1
\end{align}
and the standard $L^2$-projection $\mathcal{P}_{hp}: Q \rightarrow Q_{hp}$ with
\begin{align}\label{BBS_eq:interpolError_Php_01}
 \Vert \mymatrix{q}-\mathcal{P}_{hp}(\mymatrix{q})\Vert_{0,T}
 \lesssim \frac{h_T^{\min(p_T,t)}}{p_T^{t}} \, |\mymatrix{q}|_{t,T} \qquad 
 \forall \,  \mymatrix{q} \in H^t(T,\myspace{R}^{d\times d}) \text{ with } t\geq 0 \quad \forall \, T \in \mathcal{T}_h .
\end{align}
Operators with these approximation properties are proposed in~\cite{AINSWORTH2000329} for bilinear mappings $\mathfrak{F}_T(\cdot)$ and in \cite{Melenk2005, sanchez1984estimations}, respectively. Recall that the local polynomial degree in $Q_{hp}$ is $p_T-1$ in contrast to the polynomial degree $p_T$ in $V_{hp}$. 
Finally, we use a nodal interpolation operator $\mathcal{J}_{hp}:\prod_{T \in \mathcal{T}_h} C^0(T,\myspace{R}^{d\times d})\cap Q \rightarrow Q_{hp}$, which is the componentwise interpolation operator already defined in \eqref{BBS_eq:interpolOperator_Jhp}. By using scaling arguments, we obtain from \cite{bernardi1992polynomial} that
\begin{align} \label{BBS_eq:interpolError_Jhp}
 \Vert \mymatrix{q}-\mathcal{J}_{hp}(\mymatrix{q})\Vert_{0,T}
 \lesssim \frac{h_T^{\min(p_T,t)}}{p_T^{t}} \, |\mymatrix{q}|_{t,T} \qquad 
 \forall \, \mymatrix{q} \in H^t(T,\myspace{R}^{d\times d}) \text{ with } t>d/2 \quad \forall \, T \in \mathcal{T}_h.
\end{align}
We emphasize that by the definition of $\Lambda_{hp}$ we have $\mathcal{J}_{hp}\big( \prod_{T \in \mathcal{T}_h} C^0(T,\myspace{R}^{d\times d})\cap \Lambda \big) \subseteq \Lambda_{hp}$.
With these approximation estimates at hand we can apply density arguments to prove convergence of the discrete solution under minimal regularity conditions.

\begin{theorem}\label{BBS_thm:convergence_mixed_Lambda_hp}
There holds
\begin{align*}
 \lim_{h/p \to 0} \Big( \Vert \myvec{u}-\myvec{u}_{hp}\Vert_{1,\Omega}^2 + \Vert \mymatrix{p}-\mymatrix{p}_{hp}\Vert_{0,\Omega}^2 + \Vert\mymatrix{\lambda}-\mymatrix{\lambda}_{hp}\Vert_{0,\Omega}^2 \Big) 
 = 0.
\end{align*}
\end{theorem}

\begin{proof}
Let $\varepsilon>0$ be arbitrary. Since $C^{\infty}(\Omega, \myspace{S}_{d,0})$ is dense in $Q = L^2(\Omega, \myspace{S}_{d,0})$ there exists a $\mymatrix{p}_\varepsilon \in C^{\infty}(\Omega, \myspace{S}_{d,0})$ such that $\|\mymatrix{p}-\mymatrix{p}_\varepsilon\|_{0,\Omega} \leq \varepsilon $. From Theorem~\ref{BBS_prop:special_cases}, \eqref{eq:psi_ha_as_quadrature} and the definition of $\mathcal{Q}_{hp}(\cdot)$ and $\mathcal{J}_{hp}(\cdot)$ we obtain
\begin{align}
 (\mymatrix{\lambda}_{hp},\mathcal{J}_{hp}(\mymatrix{p}_\varepsilon) )_{0,\Omega} 
 \leq  \psi_{hp}(\mathcal{J}_{hp}(\mymatrix{p}_\varepsilon)) 
 = \mathcal{Q}_{hp}(\sigma_y \, |\mathcal{J}_{hp}(\mymatrix{p}_\varepsilon)|_F )
 = \mathcal{Q}_{hp}(\sigma_y \, |\mymatrix{p}_\varepsilon|_F).
\end{align}
By applying Corollary~\ref{BBS_lem:properties_lagrangeMult} and the Cauchy-Schwarz inequality we find that
\begin{align*}
 (\mymatrix{p}, \mymatrix{\lambda}_{hp}-\mymatrix{\lambda})_{0,\Omega} 
 &= (\mymatrix{p}-\mymatrix{p}_\varepsilon, \mymatrix{\lambda}_{hp})_{0,\Omega}  
 + (\mymatrix{p}_\varepsilon - \mathcal{J}_{hp}(\mymatrix{p}_\varepsilon), \mymatrix{\lambda}_{hp})_{0,\Omega} 
 + (\mathcal{J}_{hp}(\mymatrix{p}_\varepsilon), \mymatrix{\lambda}_{hp})_{0,\Omega} 
 - (\mymatrix{p},\mymatrix{\lambda})_{0,\Omega} \\
 &\leq \| \mymatrix{\lambda}_{hp} \|_{0,\Omega} \, \big( \| \mymatrix{p}-\mymatrix{p}_\varepsilon \|_{0,\Omega} + \| \mymatrix{p}_\varepsilon - \mathcal{J}_{hp}(\mymatrix{p}_\varepsilon) \|_{0,\Omega} \big) + \mathcal{Q}_{hp}(\sigma_y \, |\mymatrix{p}_\varepsilon|_F) - \int_\Omega \sigma_y \, |\mymatrix{p}_\varepsilon|_F \mathrm{\, d} \myvec{x} \\
 &\qquad + (\sigma_y, |\mymatrix{p}_\varepsilon|_F-|\mymatrix{p}|_F )_{0,\Omega}.
\end{align*}
From Theorem~\ref{BBS_thm:discrete_mixedF_solvable} and \eqref{BBS_eq:interpolError_Jhp} we obtain for $h/p$ being sufficiently small that
\begin{align*}
 \| \mymatrix{\lambda}_{hp} \|_{0,\Omega} \, \big( \| \mymatrix{p}-\mymatrix{p}_\varepsilon \|_{0,\Omega} + \| \mymatrix{p}_\varepsilon - \mathcal{J}_{hp}(\mymatrix{p}_\varepsilon) \|_{0,\Omega} \big) 
 \lesssim \varepsilon + \frac{h^{\min(p,2)}}{p^2} \, |\mymatrix{p}_\varepsilon|_{2,\Omega} 
 \lesssim \varepsilon.
\end{align*}
Since $\sigma_y \, |\mymatrix{p}_\varepsilon|_F \in C^0(\Omega)$ we deduce from the convergence of the Gauss quadrature (as a consequence of the Weierstrass approximation theorem) that for $h/p$ being sufficiently small there holds
\begin{align*}
 \mathcal{Q}_{hp}(\sigma_y \, |\mymatrix{p}_\varepsilon|_F) - \int_\Omega \sigma_y \, |\mymatrix{p}_\varepsilon|_F \mathrm{\, d} \myvec{x} 
 \lesssim \varepsilon.
\end{align*}
Next, we use the inverse triangle inequality to obtain
\begin{align*}
 (\sigma_y,|\mymatrix{p}_\varepsilon|_F-|\mymatrix{p}|_F )_{0,\Omega} 
 \leq \| \sigma_y \|_{0,\Omega} \, \big\| \, |\mymatrix{p}|_F-|\mymatrix{p}_\varepsilon|_F \big\|_{0,\Omega}
 \leq \| \sigma_y \|_{0,\Omega} \, \big\| \, |\mymatrix{p} - \mymatrix{p}_\varepsilon|_F \big\|_{0,\Omega} 
 = \| \sigma_y \|_{0,\Omega} \, \| \mymatrix{p} - \mymatrix{p}_\varepsilon \|_{0,\Omega} 
 \lesssim \varepsilon.
\end{align*}
Hence, combining these estimates we have
\begin{align*}
 (\mymatrix{p}, \mymatrix{\lambda}_{hp}-\mymatrix{\lambda})_{0,\Omega}
 \lesssim \varepsilon.
\end{align*}

As there exists a $ \myvec{u}_\varepsilon \in C^\infty(\Omega, \myspace{R}^d)$ with $\| \myvec{u} - \myvec{u}_\varepsilon \|_{1,\Omega} \leq \varepsilon$ and a $\mymatrix{\lambda}_\varepsilon \in H^2(\Omega,\myspace{R}^{d \times d}) \cap \Lambda$ with $\|\mymatrix{\lambda} - \mymatrix{\lambda}_\varepsilon\|_{0,\Omega} \leq \varepsilon $, see e.g.~\cite{Banz_2019OptimalControl}, choosing $\myvec{v}_{hp}=\mathcal{I}_{hp}(\myvec{u}_\varepsilon)$, $\mymatrix{q}_{hp}=\mathcal{P}_{hp}(\mymatrix{p}_\varepsilon)$, $\mymatrix{\mu}_{hp}=\mathcal{J}_{hp}(\mymatrix{\lambda}_\varepsilon)$, $\mymatrix{\mu} = \mymatrix{\lambda}$ in Theorem~\ref{BBS_prop:apriori_estimate_hild} and taking $h/p$ sufficient small finally yield
\begin{align*}
   \Vert (\myvec{u}-\myvec{u}_{hp}, \mymatrix{p}-\mymatrix{p}_{hp}) \Vert^2  + \Vert\mymatrix{\lambda}-\mymatrix{\lambda}_{hp}\Vert_{0,\Omega}^2
 & \leq  c_1\left( \Vert \myvec{u}-\myvec{v}_{hp} \Vert_{1,\Omega}^2  + \Vert \mymatrix{p}-\mymatrix{q}_{hp}  \Vert_{0,\Omega}^2 +  \Vert \mymatrix{\lambda}-\mymatrix{\mu}_{hp}\Vert_{0,\Omega}^2 \right)  \\
 & \qquad + c_2 \, (\mymatrix{\lambda}_{hp}-\mymatrix{\lambda}, \mymatrix{p})_{0,\Omega} + c_2 \,  \| \mymatrix{p} \|_{0,\Omega} \, \Vert \mymatrix{\lambda}-\mymatrix{\mu}_{hp}\Vert_{0,\Omega} \\
 & \lesssim \varepsilon,
\end{align*}
which completes the proof.
\end{proof}

As Theorem~\ref{BBS_thm:convergence_mixed_Lambda_hp} only considers norm convergence of the discrete solution there is no need to differ between the elements for which there holds $p_T=1$ and $p_T \geq 2$, respectively. As mentioned earlier (see the beginning of the proof of Theorem~\ref{thm:conver_p1} for the formal arguments) we find that $\Lambda_{hp} \subseteq \Lambda$ and $\mathcal{P}_{hp}(\Lambda) \subseteq \Lambda_{hp}$ if $p_T=1$ for all $T \in \mathcal{T}_h$. Thus, in that case Theorem~\ref{BBS_prop:apriori_estimate_hild}, $(\mymatrix{\lambda}_{hp}-\mymatrix{\lambda},\mymatrix{p})_{0,\Omega} \leq 0$ and similar arguments as in the proof of Theorem~\ref{BBS_thm:convergence_mixed_Lambda_hp} directly yield
\begin{align*}
\Vert (\myvec{u}-\myvec{u}_{hp}, \mymatrix{p}&-\mymatrix{p}_{hp}) \Vert^2 + \Vert\mymatrix{\lambda}-\mymatrix{\lambda}_{hp}\Vert_{0,\Omega}^2 \\
&\lesssim \inf_{\substack{(\myvec{v}_{hp},\mymatrix{q}_{hp},\mymatrix{\mu}_{hp}) \\ \in V_{hp}\times Q_{hp}\times \Lambda_{hp}}} \Vert (\myvec{u}-\myvec{v}_{hp}, \mymatrix{p}-\mymatrix{q}_{hp}) \Vert^2 +  \Vert \mymatrix{\lambda}-\mymatrix{\mu}_{hp}\Vert_{0,\Omega}^2 +  \Vert \mymatrix{\lambda}-\mymatrix{\mu}_{hp}\Vert_{0,\Omega} 
\xrightarrow{h \rightarrow 0} 0.
\end{align*}
As neither the Gauss quadrature nor $\mathcal{J}_{hp}(\cdot)$ is needed, we are not restricted to specific shapes of the elements $T\in\mathcal{T}_h$ in that case.

To obtain guaranteed convergence rates we need to require a certain regularity of the solution. Let us assume $(\myvec{u},\mymatrix{p}) \in H^s(\Omega,\myspace{R}^d) \times H^t(\Omega,\myspace{R}^{d\times d}) $ for some $s \geq 1$ and $t \geq 0$. From Theorem~\ref{BBS_thm:equivalence_continuous_level} we know that $\mymatrix{\lambda} = \operatorname{dev}(\mymatrix{\sigma}(\myvec{u},\mymatrix{p}) - \myspace{H}\mymatrix{p})$ and, thus, provided that $\myspace{C}$ and $\myspace{H}$ are sufficiently regular, we have
\begin{align*}
 |\mymatrix{\lambda}|_{l,\Omega} 
 &\leq |\mymatrix{\sigma}(\myvec{u},\mymatrix{p}) - \myspace{H}\mymatrix{p} |_{l,\Omega} + d^{-1} \big|\operatorname{tr}(\mymatrix{\sigma}(\myvec{u},\mymatrix{p}) - \myspace{H}\mymatrix{p})\mymatrix{I} \big|_{l,\Omega} \\
 &\lesssim |\mymatrix{\sigma}(\myvec{u},\mymatrix{p}) - \myspace{H}\mymatrix{p} |_{l,\Omega} \\
 &\leq |\myspace{C} \mymatrix{\varepsilon}(\myvec{u}) |_{l,\Omega} + |(\myspace{C}+ \myspace{H})\mymatrix{p} |_{l,\Omega} \\
 & \lesssim |\myvec{u} |_{1+l,\Omega} + |\mymatrix{p} |_{l,\Omega},
\end{align*}
which shows that $\mymatrix{\lambda}$ has at least the regularity $H^{\min\{s-1,t\}}(\Omega,\myspace{R}^{d\times d})$. Note that the Sobolev regularity of $\mymatrix{\lambda}$ might be better if the limitation of the global Sobolev regularity of $(\myvec{u},\mymatrix{p})$ does not effect the regularity of $\mymatrix{\lambda}$. This is for instance the case if a geometric singularity such as a re-entrant corner lies in the plastic region. We therefore assume $\mymatrix{\lambda}\in H^l(\Omega,\myspace{R}^{d\times d})$ for some $l\geq 0$.

\begin{theorem}\label{thm:conver_p1}
Let $p_T=1$ for all $T \in \mathcal{T}_h$ and $(\myvec{u},\mymatrix{p},\mymatrix{\lambda}) \in H^s(\Omega,\myspace{R}^d) \times H^t(\Omega,\myspace{R}^{d\times d}) \times H^l(\Omega,\myspace{R}^{d\times d})$ for some $s \geq 1$, $t\geq 0$ and $l \geq 0$. Then there holds
\begin{align*}
   \Vert (\myvec{u}-\myvec{u}_{hp}, \mymatrix{p}-\mymatrix{p}_{hp}) \Vert^2  + \Vert\mymatrix{\lambda}-\mymatrix{\lambda}_{hp}\Vert_{0,\Omega}^2 
   & \lesssim h^{2\min(1,s-1,t,l)} \left( |\myvec{u}|_{s,\Omega}^2 + |\mymatrix{p}|_{t,\Omega}^2 + |\mymatrix{\lambda}|_{l,\Omega}^2 \right).
\end{align*}
\end{theorem}

\begin{proof}
Let $\myvec{v}_{hp} = \mathcal{I}_{hp}(\myvec{u})$, $\mymatrix{q}_{hp} = \mathcal{P}_{hp}(\mymatrix{p}) $ and $\mymatrix{\mu}_{hp} = \mathcal{P}_{hp}(\mymatrix{\lambda})$. Since $p_T=1$ we find that
\begin{align*}
 \mymatrix{\mu}_{hp \, | \, T} 
 = \frac{1}{|T|} \int_T \mymatrix{\lambda} \mathrm{\, d} \myvec{x},
\end{align*}
which implies
\begin{align*}
\big| \mymatrix{\mu}_{hp \, | \, T} \big|_F 
 \leq \frac{1}{|T|} \int_T |\mymatrix{\lambda}|_F \mathrm{\, d} \myvec{x} 
 \leq \sigma_y.
\end{align*}
Therefore, it holds $\mymatrix{\mu}_{hp} \in \Lambda_{hp}$. As $\mymatrix{\lambda}_{hp} \in \Lambda_{hp}$ is piecewise constant we have that $|\mymatrix{\lambda}_{hp}|_F \leq \sigma_y$ everywhere and, thus, $\mymatrix{\lambda}_{hp} \in \Lambda$. Hence, choosing $ \mymatrix{\mu}=\mymatrix{\lambda}_{hp} $ and Theorem~\ref{BBS_prop:apriori_estimate_hild} together with \eqref{BBS_eq:interpolError_Ihp} and \eqref{BBS_eq:interpolError_Php_01} imply that
\begin{align*}
   \Vert (\myvec{u}-\myvec{u}_{hp}, \mymatrix{p}-\mymatrix{p}_{hp}) \Vert^2  + \Vert\mymatrix{\lambda}-\mymatrix{\lambda}_{hp}\Vert_{0,\Omega}^2
 & \leq  c_1 \Big( \Vert \myvec{u}-\mathcal{I}_{hp}(\myvec{u}) \Vert^2_{1,\Omega} + \Vert \mymatrix{p}-\mathcal{P}_{hp}(\mymatrix{p}) \Vert^2_{0,\Omega} +  \Vert \mymatrix{\lambda}-\mathcal{P}_{hp}(\mymatrix{\lambda})\Vert_{0,\Omega}^2 \Big)  \\
 & \qquad + c_2 \, \big(\mymatrix{\lambda}-\mathcal{P}_{hp}(\mymatrix{\lambda}), \mymatrix{p} \big)_{0,\Omega} \\
 & =  c_1 \Big( \Vert \myvec{u}-\mathcal{I}_{hp}(\myvec{u}) \Vert^2_{1,\Omega} + \Vert \mymatrix{p}-\mathcal{P}_{hp}(\mymatrix{p}) \Vert^2_{0,\Omega} +  \Vert \mymatrix{\lambda}-\mathcal{P}_{hp}(\mymatrix{\lambda})\Vert_{0,\Omega}^2 \Big)_{0,\Omega} \\
 & \qquad + c_2 \, \big( \mymatrix{\lambda}-\mathcal{P}_{hp}(\mymatrix{\lambda}), \mymatrix{p}-\mathcal{P}_{hp}(\mymatrix{p}) \big) \\
 & \leq  c_1 \, \Vert \myvec{u}-\mathcal{I}_{hp}(\myvec{u}) \Vert^2_{1,\Omega} + \frac{2c_1+c_2}{2} \Big( \Vert \mymatrix{p}-\mathcal{P}_{hp}(\mymatrix{p}) \Vert^2_{0,\Omega} +  \Vert \mymatrix{\lambda}-\mathcal{P}_{hp}(\mymatrix{\lambda})\Vert_{0,\Omega}^2 \Big) \\
 & \lesssim h^{2\min(1,s-1,t,l)} \big( |\myvec{u}|_{s,\Omega}^2 + |\mymatrix{p}|_{t,\Omega}^2 + |\mymatrix{\lambda}|_{l,\Omega}^2 \big),
\end{align*}
which completes the argument.
\end{proof}

It is not surprising that we obtain optimal order of convergence for the lowest order $h$-version as in that case the discretization is conforming; in particular, we have $\Lambda_{hp} \subseteq \Lambda$. We emphasize that we never use the Gauss quadrature in the above proof and, thus, avoid the assumption \eqref{BBS_eq:condition_detF}. Therefore, in this case there is no restriction for the shapes of elements $T\in\mathcal{T}_h$. Applying a higher-order method we need to estimate the non-conformity error $(\mymatrix{p},\mymatrix{\lambda}_{hp}-\mymatrix{\mu})_{0,\Omega}$ for $\mymatrix{\mu} \in \Lambda$. Since $\mathcal{P}_{hp}( \Lambda) \not\subset \Lambda_{hp}$ in the presents of elements $T\in\mathcal{T}_h$ with $p_T\geq 2$ we have to make use of the operator $\mathcal{J}_{hp}(\cdot)$ and therefore may not achieve optimal convergence rates.

\begin{lemma}\label{BBS_lem:consistError_02}
If $\mymatrix{p}_{ \, | \, T}\in H^t(T,\myspace{R}^{d\times d})$ and $\mymatrix{\lambda}_{ \, | \, T}\in H^l(T,\myspace{R}^{d\times d})$ for some $t$, $l>d/2$ there holds
\begin{align*}
 \int_T \mymatrix{p}: \big( \mymatrix{\lambda}_{hp}-\mymatrix{\lambda} \big) \mathrm{\, d} \myvec{x}  
 \lesssim \frac{h_T^{\min(p_T,l)}}{p_T^l} \, \Vert \mymatrix{p}\Vert_{0,T} \, |\mymatrix{\lambda}|_{l,T} + \frac{h_T^{\min(p_T,t)}}{p_T^t} \, \big(\Vert \mymatrix{\lambda}_{hp} \Vert_{0,T} + \Vert \mymatrix{\lambda}\Vert_{l,T}  \big) \, \vert \mymatrix{p}\vert_{t,T}
\end{align*}
for all $T\in\mathcal{T}_h$ with $p_T \geq 2$.
\end{lemma}

\begin{proof}
It holds
\begin{align}\label{BBS_eq:estimate_proofLem11}
 \int_T \mymatrix{p} : \big(\mymatrix{\lambda}_{hp} -\mymatrix{\lambda}\big) \mathrm{\, d} \myvec{x} 
 &= \int_T \mymatrix{p}: \big(\mathcal{J}_{hp}(\mymatrix{\lambda})-\mymatrix{\lambda} \big) + \big( \mymatrix{p}-\mathcal{J}_{hp}(\mymatrix{p}) \big): \big(\mymatrix{\lambda}_{hp}-\mathcal{J}_{hp}(\mymatrix{\lambda}) \big) + \mathcal{J}_{hp}(\mymatrix{p}) : \big(\mymatrix{\lambda}_{hp}-\mathcal{J}_{hp}(\mymatrix{\lambda}) \big)  \mathrm{\, d} \myvec{x} \nonumber \\
 &\leq \Vert \mymatrix{p}\Vert_{0,T} \, \Vert \mymatrix{\lambda}-\mathcal{J}_{hp}(\mymatrix{\lambda})\Vert_{0,T} + \Vert \mymatrix{p}-\mathcal{J}_{hp}(\mymatrix{p})\Vert_{0,T} \, \Vert \mymatrix{\lambda}_{hp}-\mathcal{J}_{hp}(\mymatrix{\lambda})\Vert_{0,T} \nonumber \\
 & \qquad + \int_T \mathcal{J}_{hp}(\mymatrix{p}):  \big(\mymatrix{\lambda}_{hp}-\mathcal{J}_{hp}(\mymatrix{\lambda}) \big)\mathrm{\, d} \myvec{x}.
\end{align}
Exploiting the representation of $\Lambda_{hp}$ in Theorem~\ref{BBS_prop:special_cases} gives the element wise estimate
\begin{align*}
\int_T \mymatrix{\lambda}_{hp} : \mathcal{J}_{hp}(\mymatrix{p}) \mathrm{\, d} \myvec{x} \leq \int_T \sigma_y  \, \mathcal{J}_{hp}(\vert\mathcal{J}_{hp}(\mymatrix{p})\vert_F) \mathrm{\, d} \myvec{x}.
\end{align*}
Since $\mymatrix{p}_{\, | \, T}\in H^t(T,\myspace{R}^{d\times d})$ and $\mymatrix{\lambda}_{ \, | \, T}\in H^l(T,\myspace{R}^{d\times d})$ for some $t$, $l>d/2$ we have $\mymatrix{p}_{\, | \, T}$, $\mymatrix{\lambda}_{\, | \, T}\in C^0(T,\myspace{R}^{d\times d})$ and, thus, there holds $\mymatrix{p} : \mymatrix{\lambda} = \sigma_y \, \vert\mymatrix{p}\vert_F$ everywhere in $T$ by Corollary~\ref{BBS_lem:properties_lagrangeMult}. 
Furthermore, as $|\det \nabla\myvec{F}_T|\in\myspace{P}_1(\widehat{T})$ the integrand is a polynomial of degree $2p_T-1$ on the reference element. Therefore, the Gauss quadrature is exact and leads to
\begin{align*}
 \int_T \mathcal{J}_{hp}(\mymatrix{p}): \big( \mymatrix{\lambda}_{hp}-\mathcal{J}_{hp}(\mymatrix{\lambda}) \big) \mathrm{\, d} \myvec{x}
 &\leq \int_T \sigma_y \, \mathcal{J}_{hp}(\vert\mathcal{J}_{hp}(\mymatrix{p})\vert_F) - \mathcal{J}_{hp}(\mymatrix{p}) : \mathcal{J}_{hp}(\mymatrix{\lambda}) \mathrm{\, d}\myvec{x} \\
 &= \sum_{k=1}^{n_T} \hat{\omega}_{k,T} \Big( \sigma_y \, \big\vert \mymatrix{p}\left(\myvec{F}_T(\hat{\myvec{x}}_{k,T})\right)\big\vert_F -  \mymatrix{p}\left(\myvec{F}_T(\hat{\myvec{x}}_{k,T})\right) :  \mymatrix{\lambda}\left(\myvec{F}_T(\hat{\myvec{x}}_{k,T})\right) \Big) \vert \det \nabla\myvec{F}_T(\hat{\myvec{x}}_{k,T})\vert \\
 &= 0.
\end{align*}
Thus, by applying \eqref{BBS_eq:interpolError_Jhp} and \eqref{BBS_eq:estimate_proofLem11} we obtain
\begin{align*}
 \int_T \mymatrix{p}: \big( \mymatrix{\lambda}_{hp}-\mymatrix{\lambda} \big) \mathrm{\, d} \myvec{x}
 &\leq \Vert \mymatrix{p}\Vert_{0,T} \, \Vert \mymatrix{\lambda}-\mathcal{J}_{hp}(\mymatrix{\lambda})\Vert_{0,T} + \Vert \mymatrix{p}-\mathcal{J}_{hp}(\mymatrix{p})\Vert_{0,T} \, \big( \Vert \mymatrix{\lambda}_{hp} \Vert_{0,T} + \Vert\mathcal{J}_{hp}(\mymatrix{\lambda})\Vert_{0,T} \big)\\
 & \lesssim \frac{h_T^{\min(p_T,l)}}{p_T^l} \, \Vert \mymatrix{p}\Vert_{0,T} \, |\mymatrix{\lambda}|_{l,T} + \frac{h_T^{\min(p_T,t)}}{p_T^t} \, \big(\Vert \mymatrix{\lambda}_{hp} \Vert_{0,T} + \Vert \mymatrix{\lambda}\Vert_{l,T}  \big) \, \vert \mymatrix{p}\vert_{t,T},
\end{align*}
which finally gives the assertion.
\end{proof}

\begin{theorem}\label{BBS_thm:convergence_rates_mixedForm}
Let $(\myvec{u},\mymatrix{p},\mymatrix{\lambda})\in H^s(\Omega,\myspace{R}^d) \times H^t(\Omega,\myspace{R}^{d\times d}) \times H^l(\Omega,\myspace{R}^{d\times d})$ for some $s\geq 1$ and $t$, $l>d/2$. Then, there holds
\begin{align*}
 \Vert (\myvec{u}-\myvec{u}_{hp}, \mymatrix{p}-\mymatrix{p}_{hp}) \Vert^2  + \Vert\mymatrix{\lambda}-\mymatrix{\lambda}_{hp}\Vert_{0,\Omega}^2 
 \lesssim \frac{h^{\min(p,2s-2,t,l)}}{p^{\min(2s-2,t,l)}}.
\end{align*}
\end{theorem}

\begin{proof}
Define $\myvec{v}_{hp} := \mathcal{I}_{hp}(\myvec{u})$, $\mymatrix{q}_{hp} := \mathcal{P}_{hp}(\mymatrix{p})$ as well as
\begin{align*}
 \mymatrix{\mu}_{hp} := \begin{cases} 
\mathcal{P}_{hp}(\mymatrix{\lambda})_{\, | \, T} , & \text{if } p_T=1, \\
\mathcal{J}_{hp}(\mymatrix{\lambda})_{\, | \, T}, & \text{otherwise}
\end{cases} \qquad \text{and} \qquad
\mymatrix{\mu} := \begin{cases} 
\mymatrix{\lambda}_{hp \, | \, T} , & \text{if } p_T=1, \\
\mymatrix{\lambda}_{\, | \, T}, & \text{otherwise.}
\end{cases}
\end{align*}
Clearly, there holds $\mymatrix{\mu}_{hp} \in \Lambda_{hp}$ and $\mymatrix{\mu} \in \Lambda$ (see the proof of Theorem~\ref{thm:conver_p1} for the case that $p_T=1$).
Then, we obtain from Theorem~\ref{BBS_prop:apriori_estimate_hild} together with \eqref{BBS_eq:interpolError_Ihp}, \eqref{BBS_eq:interpolError_Php_01} and \eqref{BBS_eq:interpolError_Jhp} that
\begin{align*}
   \Vert (\myvec{u}-\myvec{u}_{hp}, \mymatrix{p}-\mymatrix{p}_{hp}) \Vert^2  + \Vert\mymatrix{\lambda}-\mymatrix{\lambda}_{hp}\Vert_{0,\Omega}^2
 & \leq c \left( \frac{h^{2\min(p,s-1)}}{p^{2(s-1)}} \, |\myvec{u}|_{s,\Omega}^2  + \frac{h^{2\min(p,t)}}{p^{2t}} \, |\mymatrix{p}|_{t,\Omega}^2 + \frac{h^{2\min(p,l)}}{p^{2l}} \, |\mymatrix{\lambda}|_{l,\Omega}^2 \right) \\
 & \qquad + c_2 \, (\mymatrix{\lambda}_{hp}-\mymatrix{\mu} , \mymatrix{p})_{0,\Omega} + c_2 \, ( \mymatrix{\lambda}-\mymatrix{\mu}_{hp}, \mymatrix{p})_{0,\Omega}
\end{align*}
for some constant $c>0$. Moreover, the Cauchy-Schwarz inequality and Young's inequality yield together with \eqref{BBS_eq:interpolError_Php_01} and \eqref{BBS_eq:interpolError_Jhp} that
\begin{align*}
 ( \mymatrix{\lambda}-\mymatrix{\mu}_{hp}, \mymatrix{p})_{0,\Omega} 
 & =  \sum_{\substack{T\in\mathcal{T}_h \\ p_T = 1}} \int_T \big( \mymatrix{\lambda}-\mathcal{P}_{hp}(\mymatrix{\lambda}) \big) : \mymatrix{p} \mathrm{\, d} \myvec{x} + \sum_{\substack{T\in\mathcal{T}_h \\ p_T \geq 2}} \int_T \big( \mymatrix{\lambda}-\mathcal{J}_{hp}(\mymatrix{\lambda}) \big) : \mymatrix{p} \mathrm{\, d} \myvec{x} \\
 &\leq \sum_{\substack{T\in\mathcal{T}_h \\ p_T = 1}} \Vert \mymatrix{p} - \mathcal{P}_{hp}(\mymatrix{p}) \Vert_{0,T} \, \Vert \mymatrix{\lambda}-\mathcal{P}_{hp}(\mymatrix{\lambda})\Vert_{0,T} + \sum_{\substack{T\in\mathcal{T}_h \\ p_T \geq 2}} \Vert \mymatrix{p}\Vert_{0,T} \, \Vert \mymatrix{\lambda}-\mathcal{J}_{hp}(\mymatrix{\lambda})\Vert_{0,T} \\
 &\lesssim \sum_{\substack{T\in\mathcal{T}_h \\ p_T = 1}} h_T^{\min(1,t)+ \min(1,l)} \, |\mymatrix{p}|_{t,T}|\mymatrix{\lambda}|_{l,T} + \sum_{\substack{T\in\mathcal{T}_h \\ p_T \geq 2}} 
\frac{h_T^{\min(p_T,l)}}{p_T^l} \, \Vert \mymatrix{p}\Vert_{0,T} \, |\mymatrix{\lambda}|_{l,T} \\
& \lesssim  h^{2\min(1,t)} \, |\mymatrix{p}|_{t,\Omega}^2 + h^{2\min(1,l)} \, |\mymatrix{\lambda}|_{l,\Omega}^2 +
\frac{h^{\min(p,l)}}{p^{l}} \, |\mymatrix{\lambda}|_{l,\Omega}.
\end{align*}
By using the definition of $\mymatrix{\mu}$ and Lemma~\ref{BBS_lem:consistError_02} we find that
\begin{align*}
 (\mymatrix{\lambda}_{hp}-\mymatrix{\mu} , \mymatrix{p})_{0,\Omega} 
 & =  \sum_{\substack{T\in\mathcal{T}_h \\ p_T \geq 2}} \int_T \big( \mymatrix{\lambda}_{hp}-\mymatrix{\lambda} \big): \mymatrix{p} \mathrm{\, d} \myvec{x} \\
 &\lesssim \sum_{\substack{T\in\mathcal{T}_h \\ p_T \geq 2}}   \frac{h_T^{\min(p_T,l)}}{p_T^l}  \Vert \mymatrix{p}\Vert_{0,T}|\mymatrix{\lambda}|_{l,T} + \frac{h_T^{\min(p_T,t)}}{p_T^t} \, \big(\Vert \mymatrix{\lambda}_{hp} \Vert_{0,T} + \Vert \mymatrix{\lambda}\Vert_{l,T} \big) \, \vert \mymatrix{p}\vert_{t,T}\\
 & \leq \Vert \mymatrix{p}\Vert_{0,\Omega} \, \frac{h^{\min(p,l)}}{p^{l}} \, |\mymatrix{\lambda}|_{l,\Omega} + \big( \Vert  \mymatrix{\lambda}_{hp}\Vert_{0,\Omega} + \Vert \mymatrix{\lambda}\Vert_{l,\Omega} \big) \, \frac{h_T^{\min(p,t)}}{p^{t}} \, |\mymatrix{p}|_{t,\Omega}.
\end{align*}
Combining the last three estimates and Theorem~\ref{BBS_thm:discrete_mixedF_solvable} yield
\begin{align*}
   \Vert (\myvec{u}-\myvec{u}_{hp}, \mymatrix{p}-\mymatrix{p}_{hp}) \Vert^2  + \Vert\mymatrix{\lambda}-\mymatrix{\lambda}_{hp}\Vert_{0,\Omega}^2 
   & \lesssim \frac{h^{2\min(p,s-1)}}{p^{2(s-1)}} \, |\myvec{u}|_{s,\Omega}^2  +  \frac{h^{2\min(p,t)}}{p^{2t}} \, |\mymatrix{p}|_{t,\Omega}^2 + \frac{h^{2\min(p,l)}}{p^{2l}} \, |\mymatrix{\lambda}|_{l,\Omega}^2 \\
   & \qquad + \frac{h^{\min(p_T,l)}}{p^{l}} \, |\mymatrix{\lambda}|_{l,\Omega} + \frac{h^{\min(p,t)}}{p^{t}} \, |\mymatrix{p}|_{t,\Omega} \\
   & \lesssim \frac{h^{2\min(p,s-1,t,l)}}{p^{2\min(s-1,t,l)}} \, \left( |\myvec{u}|_{s,\Omega}^2  +   |\mymatrix{p}|_{t,\Omega}^2 + |\mymatrix{\lambda}|_{l,\Omega}^2 \right) \\
 & \qquad + \frac{h^{\min(p_T,t,l)}}{p^{\min(t,l)}}  \left( |\mymatrix{\lambda}|_{l,\Omega}^2 + |\mymatrix{p}|_{t,\Omega}^2 \right)^{1/2}.
\end{align*}
Finally, eliminating the dominated convergence rate terms completes the proof.
\end{proof}


\begin{figure}[ht]
  \centering 
  \begin{subfigure}[t]{0.3\textwidth}
    \centering
	\includegraphics[trim = 30mm 37mm 0mm 31mm, clip,height=43mm, keepaspectratio]{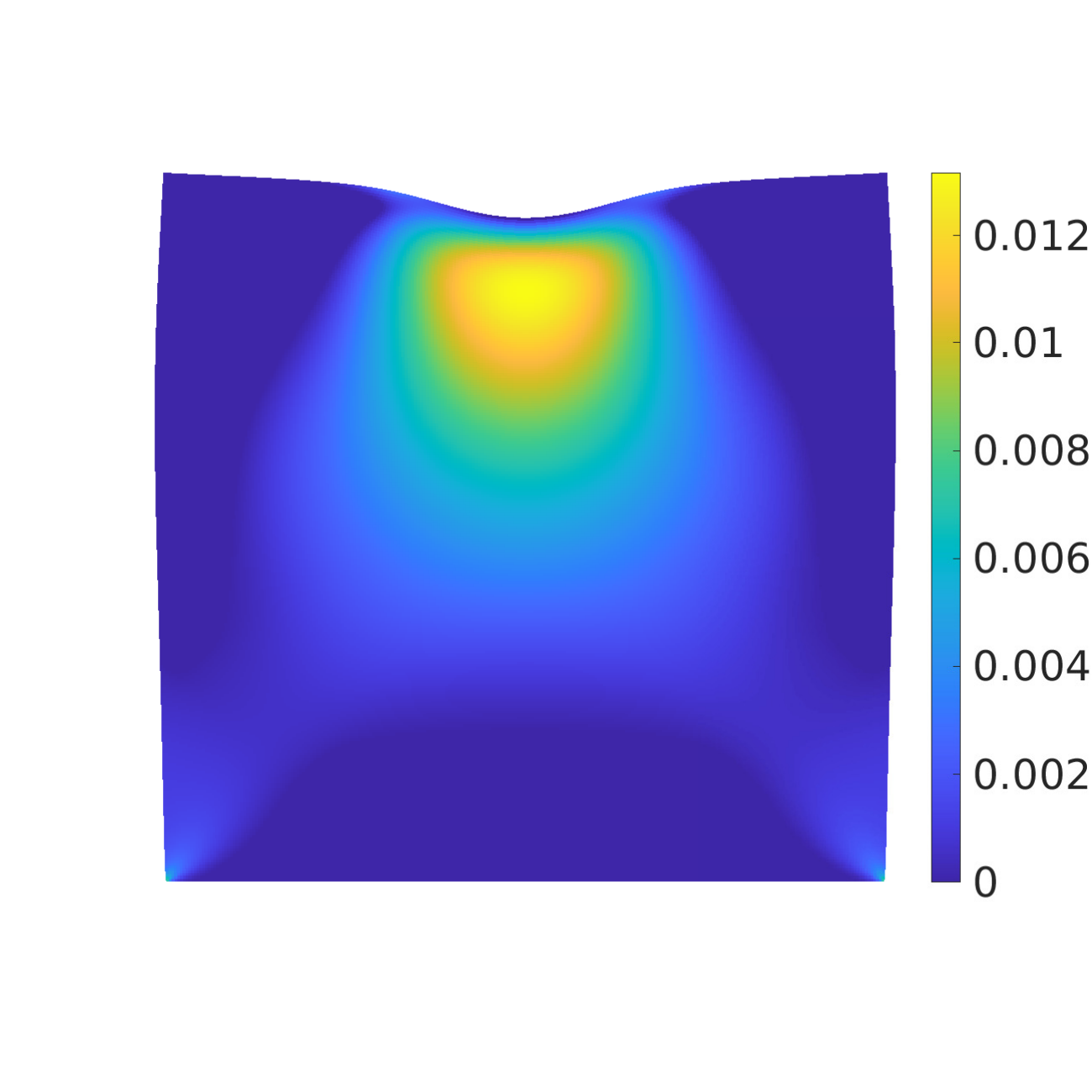}
	\caption{$\vert  \mymatrix{p}_{hp}  \vert_F$}
  \end{subfigure}
  \hspace{0.6cm} 
  \begin{subfigure}[t]{0.3\textwidth}
    \centering
	\includegraphics[trim = 30mm 37mm 15mm 31mm, clip,height=43mm, keepaspectratio]{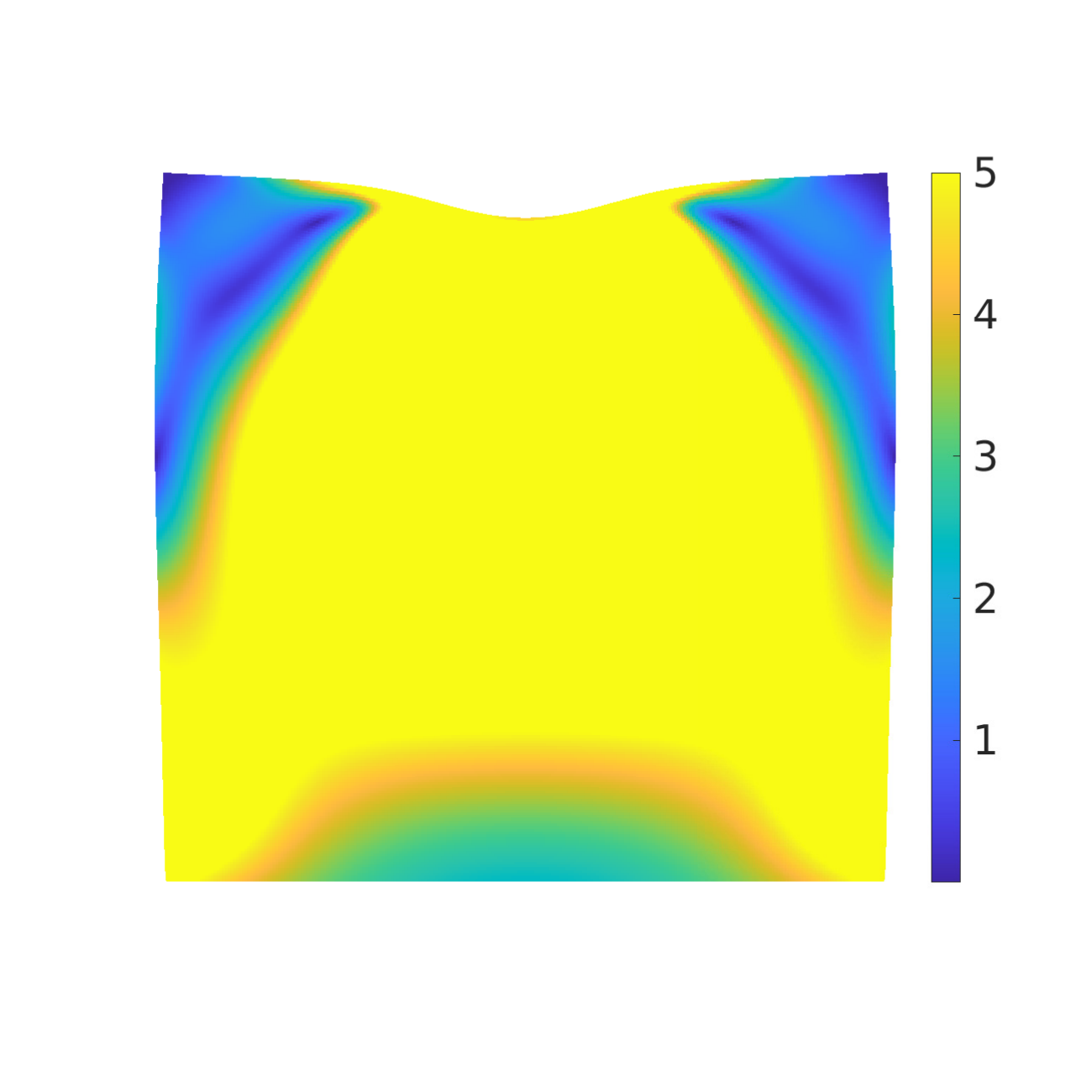}
	\caption{$\vert  \mymatrix{\lambda}_{hp} \vert_F$}
  \end{subfigure}
  \hspace{0.05cm} 
  \begin{subfigure}[t]{0.3\textwidth}
    \centering
	\includegraphics[trim = 45mm 45mm 37mm 38mm, clip,height=43mm, keepaspectratio]{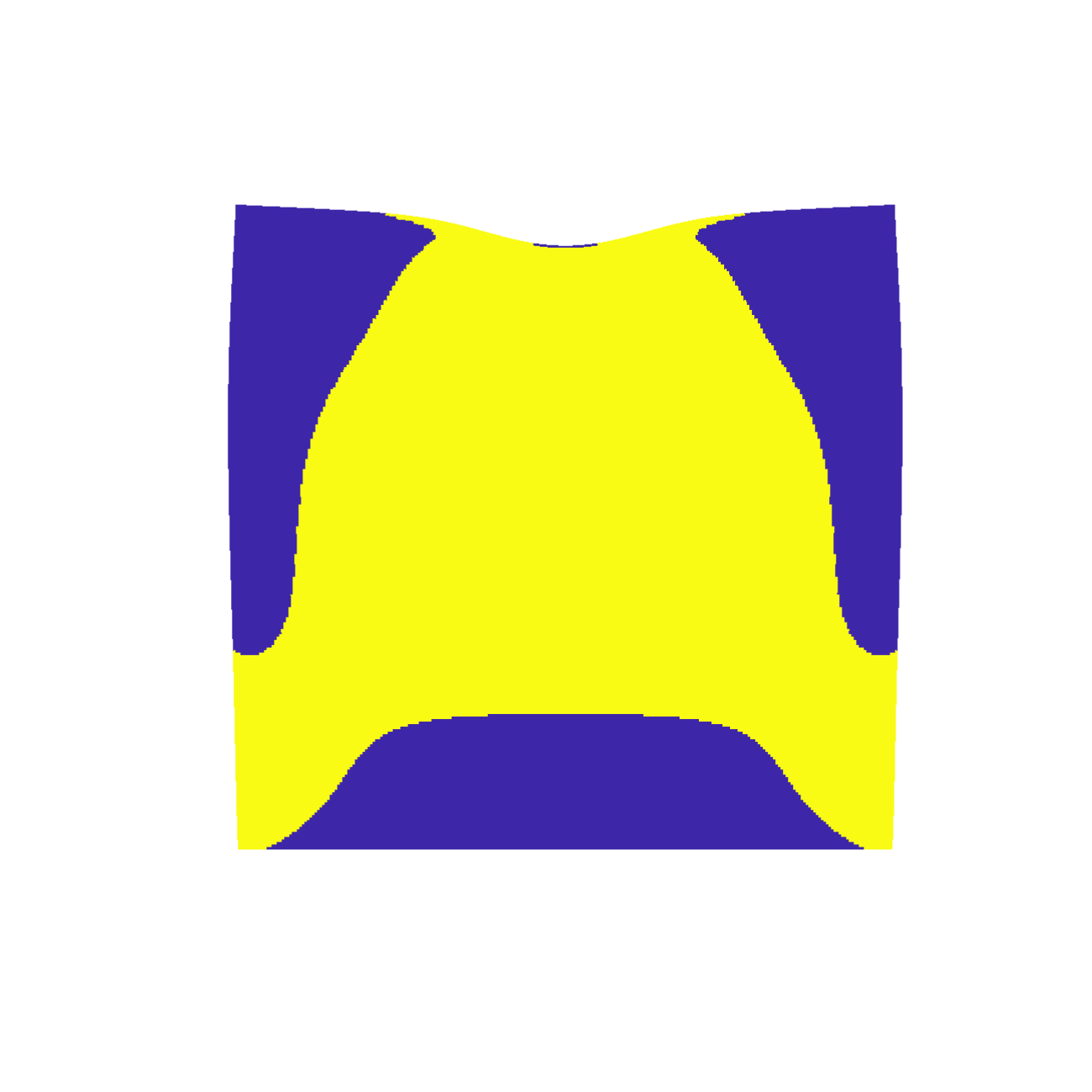}
	\caption{Purly elastic region (blue)}
	 \label{BBS_fig:elastic_region}
  \end{subfigure}
  \caption{\it Deformation of $\Omega$ magnified by factor 10 for uniform mesh with $h=2^{-7}$ and $p=1$.}
  \label{BBS_fig:solution}
\end{figure}




\section{Numerical Results}\label{BBS_sec:numerical_results}

In the following numerical experiments we choose $\Omega := (-1,1)^2$ with Dirichlet boundary $\Gamma_D := [-1,1]\times\lbrace -1 \rbrace$ and apply the volume forces $\myvec{f} := \myvec{0}$ on $\Omega$ and $\myvec{g} := (0,-400 \min(0,x_1^2-1/4)^2)^{\top}$ on $[-1,1]\times \lbrace 1\rbrace$ and zero elsewhere on the Neumann boundary. Moreover, we choose $\myspace{C}\mymatrix{\tau} := \lambda \operatorname{tr}(\mymatrix{\tau})\mymatrix{I}+2\mu\mymatrix{\tau}$ with Lam\'e constants $\lambda:=\mu:=1000$, $\myspace{H}\mymatrix{\tau}:=500\mymatrix{\tau}$ and $\sigma_y:=5$. Unfortunately, an analytic solution to that problem is unknown. Instead, a discrete solution on a very fine mesh is plotted in Figure~\ref{BBS_fig:solution}, where the displacement $\myvec{u}$ is added to the mesh nodes. The deformed body is colored according to the value of either $\vert \mymatrix{p}_{hp} \vert_F$ or $\vert \mymatrix{\lambda}_{hp} \vert_F$. In Figure~\ref{BBS_fig:elastic_region} the region of (nearly) purely elastic deformation with $\vert \mymatrix{p}_{hp}\vert_F \leq 2.22\cdot 10^{-15}$ is colored in blue and its complement in yellow. To quantify the approximation error we use 
\begin{align*}
 e_{\myvec{u}} := \| \myvec{u}_\text{fine} - \myvec{u}_{hp}\|_{1, \Omega}, \quad 
 e_{\mymatrix{p}} := \|\mymatrix{p}_\text{fine}-\mymatrix{p}_{hp} \|_{0, \Omega}, \quad
 e_{\mymatrix{\lambda}} := \| \mymatrix{\lambda}_\text{fine}- \mymatrix{\lambda}_{hp} \|_{0, \Omega},
\end{align*}
where \textit{fine} indicates an overkill discrete solution by halving $h_T$ and increasing $p_T$ by one on all elements of the finest mesh. In the following, we consider the convergence of the discretization schemes in the case of uniform $h$-refinements with $p=1,2,3$ and uniform $p$-refinements with $h=0.4$. We compare these results with those obtained by applying $h$-adaptive schemes with $p=1,2,3$ as well as with an $hp$-adaptive scheme, where the adaptivity is steered by the a posteriori error estimator
\begin{align*}
 \eta^2 
 &:= \sum_{T\in \mathcal{T}_h } \eta_T^2 +\sum_{e\in \mathcal{E}_h^I} \eta_e^2 + \sum_{e\in \mathcal{E}_h^N} \eta_{e,N}^2 + \big\Vert \operatorname{dev}(\mymatrix{\sigma}(\myvec{u}_{hp},\mymatrix{p}_{hp})-\myspace{H} \mymatrix{p}_{hp})- \mymatrix{\lambda}_{hp} \big\Vert_{0,\Omega}^2\\
 &\qquad + \Vert \mymatrix{\lambda}_{hp}-\mymatrix{\mu}^*\Vert_{0,\Omega}^2+(\sigma_y,\vert\mymatrix{p}_{hp}\vert_F)_{0,\Omega} - (\mymatrix{\mu}^*, \mymatrix{p}_{hp})_{0,\Omega}
\end{align*}
with the local contributions
\begin{align*}
 \eta_T^2 := \frac{h_T^2}{p_T^2} \, \big\Vert \myvec{f}+ \operatorname{div} \mymatrix{\sigma}(\myvec{u}_{hp},\mymatrix{p}_{hp})\big\Vert_{0,T}^2, \quad 
 \eta_e^2 := \frac{h_e}{p_e} \, \big\Vert  \llbracket \mymatrix{\sigma}(\myvec{u}_{hp},\mymatrix{p}_{hp}) \, \myvec{n}_e\rrbracket\big\Vert_{0,e}^2 , \quad 
 \eta_{e,N}^2 := \frac{h_e}{p_e} \, \big\Vert \mymatrix{\sigma}(\myvec{u}_{hp},\mymatrix{p}_{hp}) \, \myvec{n}_e-\myvec{g}\big\Vert_{0,e}^2.
\end{align*}
Here, $h_e$, $p_e$, $\mathcal{E}_h^I$, $\mathcal{E}_h^N$, $\llbracket \cdot \rrbracket$ is the edge length, edge polynomial degree, set of edges interior to $\Omega$, set of Neumann edges and the usual jump function, respectively. Furthermore, the cut-off function $\mymatrix{\mu}^*$ is defined element wise by
\begin{align*}
 \mymatrix{\mu}_{\, | \, T}^* 
 := \min\left\lbrace 1, \frac{\sigma_y}{ \; \big\vert \widehat{\mymatrix{\mu}}_{\, | \, T} \big\vert_F} \right\rbrace \, \widehat{\mymatrix{\mu}}_{\, | \, T},
 \quad \text{where} \quad
 \widehat{\mymatrix{\mu}}_{\, | \, T} 
 := \mymatrix{\lambda}_{hp \, | \, T} + \frac{1}{2} \, \mymatrix{p}_{hp \, | \, T} \qquad
 \forall \, T\in\mathcal{T}.
\end{align*}

The marking is done by the Dörfler-marking strategy with parameter $\theta=0.5$ in combination with a local regularity estimate based on the decay rate of the error estimator in $p$, see e.g.~\cite{banz2022priori,banz2020posteriori,banz2018higher}. We refer to \cite{Bammer2023Posteriori} for the introduction and detailed analysis of the above a posteriori error estimator. \\


\begin{figure}
  \centering
  \begin{subfigure}[t]{0.3\textwidth}
    \centering
	\begin{tikzpicture}[scale=0.6] 
		\begin{loglogaxis}[
			width=1.8\textwidth,
			mark size=2.5pt,
			line width=0.75pt,
			xmin=1e1,xmax=1e8,
			ymin=1e-9,ymax=2e-2,
			legend style={at={(0,0)},anchor=south west},
            legend style={fill=none}
			]
   
			\addplot+[mark=o, color=blue] table[x index=0,y index=3] {logos/h1.txt};
            \addplot+[mark=square, color=red] table[x index=0,y index=3] {logos/h2.txt};
            \addplot+[mark=x, color=green!60!black] table[x index=0,y index=3] {logos/h3.txt};
            \addplot+[mark=asterisk, color=darkbrown] table[x index=0,y index=3] {logos/p.txt};

			\addplot+[mark=+, color=purple] table[x index=0,y index=3] {logos/a1.txt};
            \addplot+[mark=o, solid, color=black] table[x index=0,y index=3] {logos/a2.txt};
            \addplot+[mark=10-pointed star, solid, color=teal] table[x index=0,y index=3] {logos/a3.txt};
            \addplot+[mark=diamond, solid, color=orange] table[x index=0,y index=3] {logos/hp.txt};

			\draw (1e6,2e-4) -- (1e7,2e-4);
			\draw (1e7,2e-4) -- (1e7,2e-4*0.464158883361278);
			\draw (1e6,2e-4) -- (1e7,2e-4*0.464158883361278);
			\node at (2.5e7,1.3e-4) {$1/3$};
   
			\draw (1e6,3e-6) -- (1e7,3e-6);
			\draw (1e7,3e-6) -- (1e7,3e-6*0.1);
			\draw (1e6,3e-6) -- (1e7,3e-6*0.1);
			\node at (1.8e7,1.2e-6) {$1$};
			
			\draw (1e5,1e-6*0.031622776601684) -- (1e6,1e-6*0.031622776601684);
 			\draw (1e5,1e-6) -- (1e5,1e-6*0.031622776601684);
 			\draw (1e5,1e-6) -- (1e6,1e-6*0.031622776601684);
			\node at (3.5e4,2e-7) {$1.5$};

   \legend{{$h1$},{$h2$},{$h3$},{$p$},{$a1$},{$a2$},{$a3$},{$hp$}}
   
		\end{loglogaxis}
	  \end{tikzpicture}
      \caption{$e_{\myvec{u}}$}
      \label{BBS_fig:errorU}
    \end{subfigure}    
    \hspace{0.3cm} 
    \begin{subfigure}[t]{0.3\textwidth}
      \centering
	  \begin{tikzpicture}[scale=0.6] 
		\begin{loglogaxis}[
			width=1.8\textwidth,
			mark size=2.5pt,
			line width=0.75pt,
			xmin=1e1,xmax=1e8,
			ymin=1e-9,ymax=2e-2,
			legend style={at={(0,0)},anchor=south west},
            legend style={fill=none}
			]
	
   			\addplot+[mark=o, color=blue] table[x index=0,y index=4] {logos/h1.txt};
            \addplot+[mark=square, color=red] table[x index=0,y index=4] {logos/h2.txt};
            \addplot+[mark=x, color=green!60!black] table[x index=0,y index=4] {logos/h3.txt};
            \addplot+[mark=asterisk, color=darkbrown] table[x index=0,y index=4] {logos/p.txt};

			\addplot+[mark=+, color=purple] table[x index=0,y index=4] {logos/a1.txt};
            \addplot+[mark=o, solid, color=black] table[x index=0,y index=4] {logos/a2.txt};
            \addplot+[mark=10-pointed star, solid, color=teal] table[x index=0,y index=4] {logos/a3.txt};
            \addplot+[mark=diamond, solid, color=orange] table[x index=0,y index=4] {logos/hp.txt};
  
			\draw (1e6,1.2e-4) -- (1e7,1.2e-4);
			\draw (1e7,1.2e-4) -- (1e7,1.2e-4*0.464158883361278);
			\draw (1e6,1.2e-4) -- (1e7,1.2e-4*0.464158883361278);
			\node at (2.5e7,8e-5) {$1/3$};
			
			\draw (1e6,3e-6) -- (1e7,3e-6);
			\draw (1e7,3e-6) -- (1e7,3e-6*0.1);
			\draw (1e6,3e-6) -- (1e7,3e-6*0.1);
			\node at (1.8e7,1.2e-6) {$1$};
			
			\draw (1e5,1e-6*0.031622776601684) -- (1e6,1e-6*0.031622776601684);
 			\draw (1e5,1e-6) -- (1e5,1e-6*0.031622776601684);
 			\draw (1e5,1e-6) -- (1e6,1e-6*0.031622776601684);
			\node at (3.5e4,2e-7) {$1.5$};

            \legend{{$h1$},{$h2$},{$h3$},{$p$},{$a1$},{$a2$},{$a3$},{$hp$}}
   
		\end{loglogaxis}
	   \end{tikzpicture}
       \caption{$e_{\mymatrix{p}}$}
       \label{BBS_fig:errorP}
     \end{subfigure}
    \hspace{0.3cm} 
    \begin{subfigure}[t]{0.3\textwidth}
      \centering
	  \begin{tikzpicture}[scale=0.6] 
		\begin{loglogaxis}[
			width=1.8\textwidth,
			mark size=2.5pt,
			line width=0.75pt,
			xmin=1e1,xmax=1e8,
			ymin=1e-6,ymax=1e1,
			legend style={at={(0,0)},anchor=south west},
            legend style={fill=none}
			]

   			\addplot+[mark=o, color=blue] table[x index=0,y index=5] {logos/h1.txt};
            \addplot+[mark=square, color=red] table[x index=0,y index=5] {logos/h2.txt};
            \addplot+[mark=x, color=green!60!black] table[x index=0,y index=5] {logos/h3.txt};
            \addplot+[mark=asterisk, color=darkbrown] table[x index=0,y index=5] {logos/p.txt};

			\addplot+[mark=+, color=purple] table[x index=0,y index=5] {logos/a1.txt};
            \addplot+[mark=o, solid, color=black] table[x index=0,y index=5] {logos/a2.txt};
            \addplot+[mark=10-pointed star, solid, color=teal] table[x index=0,y index=5] {logos/a3.txt};
            \addplot+[mark=diamond, solid, color=orange] table[x index=0,y index=5] {logos/hp.txt};
  
			\draw (1e6,9e-2) -- (1e7,9e-2);
			\draw (1e7,9e-2) -- (1e7,9e-2*0.316227766016838);
			\draw (1e6,9e-2) -- (1e7,9e-2*0.316227766016838);
			\node at (2.5e7,5.5e-2) {$0.5$};
			
			\draw (1e6,2e-3) -- (1e7,2e-3);
			\draw (1e7,2e-3) -- (1e7,2e-3*0.1);
			\draw (1e6,2e-3) -- (1e7,2e-3*0.1);
			\node at (1.8e7,8e-4) {$1$};
			
			\draw (1e5,2e-3*0.031622776601684) -- (1e6,2e-3*0.031622776601684);
 			\draw (1e5,2e-3) -- (1e5,2e-3*0.031622776601684);
 			\draw (1e5,2e-3) -- (1e6,2e-3*0.031622776601684);
			\node at (3.5e4,3.5e-4) {$1.5$};

            \legend{{$h1$},{$h2$},{$h3$},{$p$},{$a1$},{$a2$},{$a3$},{$hp$}}
      
		\end{loglogaxis}
	  \end{tikzpicture}
      \caption{$e_{\mymatrix{\lambda}}$}
      \label{BBS_fig:errorL}
    \end{subfigure}

   \caption{\it Individual approximation errors vs.~degrees of freedom.
   Legend: $hi$ stands for uniform $h$-refinement with $p=i$, $p$ for uniform $p$-refinement with $h=0.4$, $ai$ for $h$-adaptive scheme with $p=i$ and $hp$ for $hp$-adaptive scheme.}
  \label{BBS_fig:error}
\end{figure}
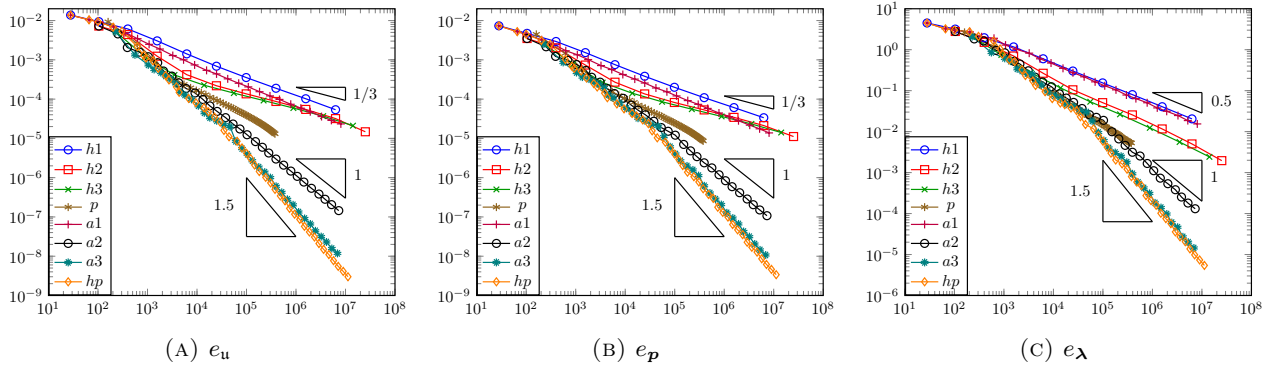


Figure~\ref{BBS_fig:error} shows the reduction of the three errors $e_{\myvec{u}}$, $e_{\mymatrix{p}}$ and $e_{\mymatrix{\lambda}}$ separately for the different discretization schemes. The observed experimental order of convergence (EOC) with respect to the degrees of freedom $N$ are collected in Table~\ref{tab:EOC}. First of all, all methods converge, which is in agreement with Theorem~\ref{BBS_thm:convergence_mixed_Lambda_hp} for the uniform methods. Secondly, the EOC of the method based on the uniform $h$-refinement with $p=1$ is 0.43 and in agreement with Theorem~\ref{thm:conver_p1}. In view of the Dirichlet-to-Neumann (DtN) transition singularity in the lower left and right corner the EOC is still better than asymptotically expected. The EOC of the other two methods based on uniform $h$-refinements with $p=2,3$ is $1/3$ and therefore in agreement with Theorem~\ref{BBS_thm:convergence_rates_mixedForm}. As the EOC of the method based on uniform $p$-refinements is roughly $2/3$ we get an indication that the DtN-singularities in the two corners, for which this method has a doubled EOC compared to the those using uniform $h$-refinements, are stronger than the singularity at the free boundary separating the purely elastic deformed region from the region with plastic deformations, see Figure~\ref{BBS_fig:elastic_region}. 

Applying $h$-adaptive refinements we obtain optimal EOC, namely $1/2$, $1$ and $3/2$ for $p=1$, $p=2$ and $p=3$, respectively. Increasing $p$ even further does not increase the EOC  as isotropic refinements have limited capabilities to resolve the singularity of the curved free boundary. The best EOC is obtained by the $hp$-adaptive scheme.


\begin{table}[ht]
\centering
\begin{tabularx}{\textwidth}{X | X X X X X X X X}
&  $h1$ & $h2$ & $h3$ & $p$ & $a1$ & $a2$ & $a3$ & $hp$    \\  \hline
$e_{\myvec{u}}$            & 0.46 & 0.34 & 0.33 & 0.70 & 0.5 & 1.0 & 1.41 & 1.55 \\ 
$e_{\mymatrix{p}}$         & 0.43 & 0.33 & 0.33 & 0.68 & 0.5 & 1.0 & 1.45 & 1.50 \\ 
$e_{\mymatrix{\lambda}}$   & 0.49 & 0.55 & 0.54 & 0.77 & 0.5 & 1.1 & 1.48 & 1.46
\end{tabularx} 
\vspace{0.25cm}
\caption{\it Experimental order of convergence (EOC).}
\label{tab:EOC}
\end{table}


\pagebreak

\bibliographystyle{amsalpha}

\vspace*{-0.25cm}

\begin{thebibliography}{9}


\bibitem{AINSWORTH2000329}
M.~Ainsworth and D.~Kay, Approximation theory for the $hp$-version finite element method and application to the non-linear laplacian, \emph{Appl.~Numer.~Math.} \textbf{34} (2000) 329–344.


\bibitem{Bammer2022Icosahom}
P.~Bammer, L.~Banz and A.~Schr\"{o}der, $hp$-Finite Elements with Decoupled Constraints for Elastoplasticity, \emph{Spectral and High Order Methods for Partial Differential Equations ICOSAHOM 2020+ 1, Springer} (2023) 141–153.

\bibitem{Bammer2023Posteriori}
P.~Bammer, L.~Banz and A.~Schr\"{o}der, A Posteriori Error Estimates for $hp$-FE Discretizations in Elastoplasticity, \emph{arXiv preprint} (2023).

\bibitem{banz2022priori}
L.~Banz, O.~Hern\'{a}ndez and E.P.~Stephan, A priori and a posteriori error estimates for $hp$-FEM for a Bingham type variational inequality of the second kind, \emph{Comput. Math. Appl.} \textbf{126} (2022) 14–30.

\bibitem{banz2020posteriori}
L.~Banz, M.~Hinterm\"{u}ller and A.~Schr\"{o}der, A posteriori error control for distributed elliptic optimal control problems with control constraints discretized by $hp$-finite elements, \emph{Comput. Math. Appl.} \textbf{80} (2020) 2433–2450.

\bibitem{Banz_2019OptimalControl}
L.~Banz, M.~Hinterm\"{u}ller and A.~Schr\"{o}der, $hp$-finite elements for elliptic optimal control problems with control constraints, in preparation.

\bibitem{banz2018higher}
L.~Banz, B.P.~Lamichhane and E.P.~Stefan, Higher order FEM for the obstacle problem of the $p$-Laplacian -- variational inequality approach, \emph{Comput.~Math.~Appl.} \textbf{76} (2018) 1639-1660.

\bibitem{banz2019hybridization}
L.~Banz, J.~Petsche and A.~Schr\"{o}der, Hybridization and stabilization for $hp$-finite element methods, \emph{Appl.~Numer.~Math.} \textbf{136} (2019) 66-102.

\bibitem{Banz2021abstract}
L.~Banz and A.~Schr\"{o}der, A posteriori error control for variational inequalities with linear constraints in an abstract framework, \emph{J.~Appl.~Numer.~Optim} \textbf{3} (2021) 333--359.

\bibitem{bernardi1992polynomial}
C.~Bernardi, Y.~Maday, Polynomial interpolation results in Soblolev spaces, \emph{J.~Comput.~Appl.~Math.} \textbf{43} (1992) 53-80.

\bibitem{brokate2005quasi}
M.~Brokate, C.~Carstensen and J.~Valdman, A quasi-static boundary value problem in multi-surface elastoplasticity: part 2 -- numerical solution, \emph{Math.~Methods Appl.~Sci.} \textbf{28} (2005) 881-901.

\bibitem{burg2015posteriori}
M.~B\"{u}rg and A.~Schr\"{o}der, A posteriori error control for $hp$-finite elements for variational inequalities of the first and second kind, \emph{Comput.~Math.~Appl.} \textbf{70} (2015) 2783-2802.

\bibitem{byfut2017unsymmetric}
A.~Byfut and A.~Schr\"{o}der, Unsymmetric multi-level hanging nodes and anisotropic polynomial degrees in $h1$-conforming higher-order finite element methods, \emph{Comput.~Math.~Appl.} \textbf{73} (2017) 2092-2150.


\bibitem{carstensen1999numerical}
C.~Carstensen, Numerical analysis of the primal problem of elastoplasticity with hardening, \emph{Numer.~Math.} \textbf{82} (1999) 577-597.

\bibitem{carstensen2006reliable}
C.~Carstensen, R.~Klose and A.~Orlando, Reliable and efficient equilibrated a posteriori finite element error control in elastoplasticity and elastoviscoplasticity with hardening, \emph{Comput.~Methods Appl.~Mech.~Engrg.} \textbf{195} (2006) 2574-2598.

\bibitem{zbMATH06537874}
C.~Carstensen, A.~Schr\"{o}der and S.~Wiedemann, An optimal adaptive finite element method for elastoplasticity, \emph{Numer.~Math.} \textbf{132} (2016) 131-154.

\bibitem{chen1988plasticity}
W.F.~Chen and D.J.~Han, \emph{Plasticity for Structural Engineers}, Springer, 1988.

\bibitem{christensen2002nonsmooth}
P.W.~Christensen, A nonsmooth Newton method for elasoplastic problems, \emph{Comp.~Methods Appl.~Mech.~Engrg.} \textbf{191} (2002) 1189-1219.


\bibitem{Ekeland_1976}
I.~Ekeland and R.~T\'{e}mam, \emph{Convex analysis and variational problems}, North-Holland Publishing Company, 1976.


\bibitem{han1991finite}
W.~Han, Finite element analysis of a holonomic elastic-plastic problem, \emph{Numer.~Math.} \textbf{60} (1991) 493-508.

\bibitem{han1995finite}
W.~Han and B.D.~Reddy, On the finite element method for mixed variational inequalities arising in elastoplasticity, \emph{SIAM J.~Numer.~Anal.} \textbf{32} (1995) 1778-1807.

\bibitem{Han2013}
W.~Han and B.D.~Reddy, \emph{Plasticity. Mathematical Theory and Numerical Analysis}, Springer, 2 edition, 2013.

\bibitem{hild2002quadratic}
P.~Hild and P.~Laborde, Quadratic finite element methods for unilateral contact problems, \emph{Appl.~Numer.~Math.} \textbf{41} (2002) 401-421.


\bibitem{Luksan1998}
L.~Luk\v{s}an and J.~Vl\v{c}ek, A bundle-Newton method for nonsmooth unconstrained minimization, \emph{Math.~Program.} \textbf{83} (1998) 373-391.


\bibitem{Melenk2005}
J.M.~Melenk, $hp$-interpolation of nonsmooth functions and an application to $hp$-a posteriori error estimation, \emph{SIAM J.~Numer.~Anal.} \textbf{43} (2005) 127-155.


\bibitem{ovcharova2017coupling}
N.~Ovcharova and L.~Banz, Coupling regularization and adaptive $hp$-BEM for the solution of a delamination problem, \emph{Numer.~Math.} \textbf{137} (2017) 303-337.


\bibitem{reddy1987variational}
B.D.~Reddy and T.B.~Griffin, Variational principles and convergence of finite element approximations of a holonomic elastic-plastic problem, \emph{Numer.~Math.} \textbf{52} (1987) 101-117.


\bibitem{sanchez1984estimations}
A.~Sanchez and R.~Arcangeli, Estimations des erreurs de meilleure approximation polynomiale et d'interpolation de Lagrange dans les espaces de Sobolev d'ordre non entier, \emph{Numer. Math.} \textbf{45} (1984) 301-321.

\bibitem{schroder2011mixed}
A.~Schr\"{o}der, Mixed FEM of higher-order for a frictional contact problem, \emph{PAMM} \textbf{11} (2011) 7-10.

\bibitem{zbMATH05872978}
A.~Schr\"{o}der, H.~Blum, A.~Rademacher and H.~Kleemann, Mixed FEM of higher order for contact problems with fiction, \emph{Int.~J.~Numer.~Anal.~Model.} \textbf{8} (2011) 302-323.

\bibitem{schroder2011error}
A.~Schr\"{o}der and S.~Wiedemann, Error estimates in elastoplasticity using a mixed method, \emph{Appl.~Numer.~Math.} \textbf{61} (2011) 1031-1045.


\bibitem{wiedemann2013Adaptive}
S.~Wiedemann, \emph{Adaptive finite elements for a contact problem in elastoplasticity with Lagrange techniques}, Ph.D.~thesis, Humboldt-Universit\"{a}t zu Berlin, Mathematisch-Naturwissenschaftliche Fakult\"{a}t II, 2013.


\end{thebibliography}

\end{document}